\documentclass[preprint]{elsarticle}

\usepackage{algorithm}
\usepackage{algorithmic}
\usepackage{amsmath}
\usepackage{amssymb}
\usepackage{amsthm}
\usepackage{hyperref}
\usepackage{rotating}
\usepackage{tikz}
\usepackage{url}

\usetikzlibrary{arrows,calc,decorations.markings}

\tikzset{bigah/.style={decoration={markings,
  mark=at position 1 with {\arrow[ultra thick]{latex'}}},
  postaction={decorate},shorten >=1.4pt}}

\newtheorem{theorem}{Theorem}[section]
\newtheorem{lemma}[theorem]{Lemma}
\newtheorem{corollary}[theorem]{Corollary}
\theoremstyle{definition}
\newtheorem{conjecture}[theorem]{Conjecture}
\newtheorem{definition}[theorem]{Definition}

\DeclareMathOperator{\perm}{perm}

\newcommand*{\eclipse}{ECL$^i$PS$^e$}

\journal{Discrete Mathematics}

\begin{document}
\begin{frontmatter}


\title{Cycle-maximal triangle-free graphs\tnoteref{tn1}}
\tnotetext[tn1]{NOTICE: this is the authors' version of a work that was
accepted for publication in \emph{Discrete Mathematics}.  Changes resulting
from the publishing process, such as peer review, editing, corrections,
structural formatting, and other quality control mechanisms may not be
reflected in this document.  Changes may have been made to this work since
it was submitted for publication.  A definitive version was subsequently
published in \emph{Discrete Mathematics}, Volume 338, Issue 2 (6 February
2015), pages 274-290.  DOI:
\href{http://dx.doi.org/10.1016/j.disc.2014.10.002}{10.1016/j.disc.2014.10.002}}


\author[compsci,nserc]{Stephane Durocher}
\ead{durocher@cs.umanitoba.ca}

\author[math,nserc]{David S. Gunderson}
\ead{gunderso@cc.umanitoba.ca}

\author[compsci]{Pak Ching Li}
\ead{lipakc@cs.umanitoba.ca}

\author[compsci]{Matthew~Skala\corref{cor1}}
\ead{mskala@ansuz.sooke.bc.ca}
\cortext[cor1]{Principal corresponding author.}

\address[compsci]{Department of Computer Science,
E2--445 EITC, University of Manitoba, Winnipeg, Manitoba, Canada, R3T 2N2}

\address[math]{Department of Mathematics,
342 Machray Hall, 186 Dysart Road,
University of Manitoba, Winnipeg, Manitoba, Canada, R3T 2N2}

\fntext[nserc]{Work of these authors is supported in part by the Natural
Sciences and Engineering Research Council of Canada (NSERC).}



\begin{abstract}
We conjecture that the balanced complete bipartite graph $K_{\lfloor n/2
\rfloor,\lceil n/2 \rceil}$ contains more cycles than any other $n$-vertex
triangle-free graph, and we make some progress toward proving this.  We give
equivalent conditions for cycle-maximal triangle-free graphs; show bounds on
the numbers of cycles in graphs depending on numbers of vertices and edges,
girth, and homomorphisms to small fixed graphs; and use the bounds to show
that among regular graphs, the conjecture holds.  We also
consider graphs that are close to being regular, with the minimum and
maximum degrees differing by at most a positive integer $k$.  For $k=1$, we
show that any such counterexamples have $n\le 91$ and are not homomorphic
to $C_5$; and for any fixed $k$ there exists a finite upper bound on the
number of vertices in a counterexample.  Finally, we describe an algorithm
for efficiently computing the matrix permanent (a $\#P$-complete
problem in general) in a special case used by our bounds.
\end{abstract}


\end{frontmatter}



\section{Introduction}
\label{sec:intro}


Many algorithmic problems that are computationally difficult on graphs can
be solved easily in polynomial time when the graph is acyclic.
Limiting input to trees (connected acyclic graphs) or forests (acyclic
graphs), however, is often too restrictive; many of these problems remain
efficiently solvable when the graph is ``nearly'' a tree
\cite{Arnborg:Easy,Bern:Linear,Bodlaender:Dynamic,Gurevich:Solving}. 
Various notions exist formalizing how close a given graph is to being a
tree, including bounded treewidth (partial $k$-trees), $k$-connectivity, and
number of cycles. 

The problem of evaluating $c(G)$ for a given graph is $\#P$-complete,
equivalent in difficulty to counting the certificates of an $NP$-complete
decision problem, even though the problem of testing for the existence of a
single cycle is trivially polynomial-time.
Existence of a cycle is a graph
property definable in monadic second-order logic.
By the result known as
Courcelle's Theorem~\cite{Courcelle:Monadic}, such properties can be decided
in linear time for graphs of bounded treewidth, and as described by Arnborg,
Lagergren, and Seese, the counting versions are also linear-time for fixed
treewidth~\cite{Arnborg:Easy}.
On the other hand, if we parameterize by
length of the cycles instead of structure of the graph, Flum and
Grohe~\cite{Flum:Parameterized} give evidence against fixed-parameter
tractability: they show that counting cycles of length $k$ is
$\#W[1]$-complete, with no $(f(k)\cdot n^c)$-time algorithm unless the
Exponential Time Hypothesis fails.

When no restrictions are imposed on the graph, the
number of cycles in an $n$-vertex graph is maximized by the complete graph
on $n$ vertices, $K_n$.  In this case the number of cycles is easily seen to
be
\begin{equation}
  \sum_{i=3}^n \left( \binom{n}{i} \frac{(i-1)!}{2}
    \right) = n!  \sum_{i=3}^n \frac{1}{2i(n-i)!} \, .  \label{eqn:kn}
\end{equation}

The bound \eqref{eqn:kn} can be refined by introducing additional
parameters.  Previous results include bounds on the number of cycles in
terms of $n$, $m$, $\delta$, and $\Delta$ (the number of vertices, number of
edges, minimum degree, and maximum degree of $G$, respectively)
\cite{Entringer:Maximum,Guichard:Maximum,Volkmann:Estimations}, as well as
bounds on the number of cycles for various classes of graphs, including
$k$-connected graphs \cite{Knor:Number}, Hamiltonian graphs
\cite{Rautenbach:Maximum,Shi:Number}, planar graphs
\cite{Aldred:Maximum,Alt:Number}, series-parallel graphs
\cite{DeMier:Maximum}, and random graphs \cite{Takacs:Limit}.

A graph's cycles can be classified by length.  For each value of $i$, the
summand in \eqref{eqn:kn} corresponds to the number of cycles of length $i$
in $K_n$.  If short cycles are disallowed, the number of long cycles
possible is also reduced.  Every graph $G$ of girth $g$ that contains two or
more cycles has $n \geq 3g/2-1$ vertices or, equivalently, if $g >
2(n+1)/3$, then $G$ has at most one cycle \cite{Bose:Bounding}.  The bound
on the number of cycles increases as $g$ decreases.  As mentioned earlier,
the case $g=3$ is maximized by $K_n$ for which the number of cycles is
exactly \eqref{eqn:kn}.  Can the maximum number of cycles be expressed
exactly or bounded tightly as a function of arbitrary values for $n$ and
$g$?  Even when $g=4$ the maximum number of possible cycles is
unknown.  Graphs of girth four or greater are exactly the triangle-free
graphs.  One goal of this research program is to show that the number of
cycles in a triangle-free $n$-vertex graph is maximized by the complete
bipartite graph $K_{\lfloor n/2\rfloor, \lceil n/2\rceil}$, and the results
in this paper represent significant progress toward that goal.

We first encountered the problem of bounding the number of cycles as a
function of $n$ and $g$ when examining path-finding algorithms on graphs. 
A tree traversal can be achieved by applying a right-hand rule (e.g.,
after reaching a vertex $v$ via its $i$th edge, depart along its $(i+1)$st
edge). Traversing a graph using only local information at each vertex
is significantly more difficult in graphs with cycles.
A successful traversal can be guaranteed, however, if the local
neighbourhood of every vertex $v$ is tree-like within some distance $k$
from $v$ (e.g., the graph has girth $g \geq 2k+1$) and that a fixed upper bound 
is known on the number of possible cycles along paths that join pairs of leaves
outside each such local tree (Bose, Carmi, and Durocher~\cite{Bose:Bounding}
give a more formal discussion).  Deriving a useful bound on this number of
cycles led to the work presented in this paper.

In any graph, every chordless cycle of length seven or greater can be bridged
by the addition of a chord without creating any triangles.
Similarly, in any graph of girth six or greater, any given 
cycle can be bridged without creating any triangles.
There exist graphs of girth four and five, 
however, that contain cycles of length 
six that cannot be bridged without creating a triangle.  The Petersen graph
minus one vertex, as shown in Figure~\ref{fig:pminus},
is such a graph of girth five;
replacing one of its vertices with two sharing the same
neighbourhood results in a graph of girth four with the same property.
To increase the number of cycles in a graph, 
large chordless cycles can be bridged greedily 
until the graph is triangle-free but the addition of any edge would create
a triangle.
This suggests that a cycle-maximal triangle-free graph 
should contain many cycles of length four or five.
Since bipartite graphs are triangle free,
complete bipartite graphs and, more specifically, balanced bipartite
graphs are natural candidates for maximizing
the number of cycles.
We verified the following conjecture to be true 
by exhaustive computer search for $n \le 13$:

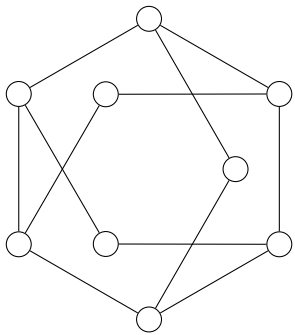
\begin{figure}
\centering\begin{tikzpicture}
  \node[draw,circle] (a) at (30:2cm) {~};
  \node[draw,circle] (b) at (90:2cm) {~} edge (a);
  \node[draw,circle] (c) at (150:2cm) {~} edge (b);
  \node[draw,circle] (d) at (210:2cm) {~} edge (c);
  \node[draw,circle] (e) at (270:2cm) {~} edge (d);
  \node[draw,circle] (f) at (330:2cm) {~} edge (e) edge (a);
  \node[draw,circle] (g) at (0:1.15cm) {~} edge (b) edge (e);
  \node[draw,circle] (h) at (120:1.15cm) {~} edge (a) edge (d);
  \node[draw,circle] (i) at (240:1.15cm) {~} edge (c) edge (f);
\end{tikzpicture}
\caption{The Petersen graph minus one vertex, which contains a $C_6$ that
  cannot be bridged without creating a triangle.}
\label{fig:pminus}
\end{figure}

\begin{conjecture}\label{con:main}
The cycle-maximal triangle-free graphs are exactly the bipartite Tur\'{a}n
graphs, $K_{\lfloor n/2 \rfloor, \lceil n/2 \rceil}$ for all $n$.
\end{conjecture}


\subsection{Overview of results}

Our main results, Theorems~\ref{thm:regular-triangle-free}
and~\ref{thm:near-reg-triangle-free}, show that Conjecture~\ref{con:main}
holds for all regular cycle-maximal triangle-free graphs, and all
near-regular cycle-maximal triangle-free graphs with greater than $91$
vertices.  In Section~\ref{sec:props} we give some properties of
cycle-maximal graphs.  In Section~\ref{sec:bounds} we establish bounds on
the number of cycles in triangle-free graphs.  In Section~\ref{sec:regular}
we prove Theorem~\ref{thm:regular-triangle-free}, and in
Section~\ref{sec:near-regular} we prove
Theorem~\ref{thm:near-reg-triangle-free}.  Section~\ref{sec:algorithm}
describes an algorithm for computing the matrix permanent, which is used in
our bounds.


\subsection{Definitions and notation}\label{sub:definitions}

Graphs are simple and undirected unless otherwise specified.
A \emph{block} in a graph $G$ is a maximal 2-connected subgraph of $G$.
Given a graph
$G$, let $V(G)$, $E(G)$, $\delta(G)$, and $\Delta(G)$ denote,
respectively, the vertex set of $G$, edge set of $G$, minimum degree of any
vertex in $G$, and maximum degree of any vertex in $G$.  Given a vertex $v
\in V(G)$, let $N(v)$ denote the neighbourhood of $v$; that is, the set of all
vertices adjacent to $v$ in $G$.  Given positive integers $s$ and $t$, let
$K_s$ denote the complete graph on $s$ vertices, $K_{s,t}$ denote the complete
bipartite graph with part sizes $s$ and $t$, $C_s$ denote the cycle of $s$
vertices, and $P_s$ denote the path of $s$ vertices.
Given a positive integer $n$, let $T(n,2)$ represent the
bipartite Tur\'{a}n graph on $n$ vertices, that is, $K_{\lfloor n/2 \rfloor,
\lceil n/2 \rceil}$.

A graph is \emph{triangle-free} if it does not contain $C_3$ (a
\emph{triangle}) as a subgraph.  The \emph{girth} of a graph is the size of
the smallest cycle, by convention $\infty$ if there are no cycles.
Triangle-free is equivalent to having girth at least 4.  A graph $G$ is
\emph{maximal triangle-free} if it is triangle-free, but adding any edge
would create a triangle.  Let $c(G)$ denote the number of labelled cycles in
$G$.  That is the number of distinct subsets of $E(G)$ that are cycles; note
that we are not only counting distinct cycle \emph{lengths}, which may also be
interesting but is a completely different problem.  Then $G$ is
\emph{cycle-maximal} for some class of graphs and number of vertices $n$ if
$G$ maximizes $c(G)$ among $n$-vertex graphs in the class.  Most often we
are interested in cycle-maximal graphs for fixed minimum girth $g$,
especially the case $g\ge 4$, \emph{cycle-maximal triangle-free} graphs.  It
is easy to prove (see Lemma~\ref{lem:every-edge-cfour}) that a cycle-maximal
triangle-free graph, if large enough to have any cycles at all, is also
maximal triangle-free.

A graph $G$ is \emph{homomorphic} to a graph $H$ when there exists a
function $f:V(G) \rightarrow V(H)$, called a \emph{homomorphism}, such that
if $(u,v) \in E(G)$ then $(f(u),f(v)) \in E(H)$.  A graph is $s$-colourable
if and only if it is homomorphic to $K_s$.  Given a positive integer $t$ and
a graph $H$, let $H(t)$ represent the uniform \emph{blowup} of $H$: that is
the graph homomorphic to $H$ formed by replacing the vertices in $H$ with
independent sets, each of size $t$, and adding edges between all vertices in
two independent sets if the sets correspond to adjacent vertices in $H$.  If
$H$ has $p$ vertices, then $H(t)$ has $pt$ vertices.  When $H$ is a labelled
graph with $p$ vertices $v_1,v_2,\ldots,v_p$, let $H(n_1,n_2,\ldots,n_p)$
represent the not necessarily uniform blowup of $H$ in which $v_1$ is
replaced by an independent set of size $n_1$, $v_2$ by an independent set of
size $n_2$, and so on, with all edges added that are allowed by the
homomorphism.

We define the family of \emph{gamma graphs} as follows.  For any positive
integer $i$, $\Gamma_i$ is a graph with $n=3i-1$ vertices
$v_1,v_2,\ldots,v_n$.  Each vertex $v_j$ is adjacent to the $i$ vertices
$v_{j+i},v_{j+i+1},\ldots,v_{j+2i-1}$, taking the indices modulo $n$.  For
$i\ge2$, this is the complement of the $(i-1)$st power of the cycle graph
$C_{3i-1}$.  Then $\Gamma_1$ is $K_2$, $\Gamma_2$ is $C_5$, and $\Gamma_3$
is the eight-vertex M\"{o}bius ladder, or twisted cube, shown in
Figure~\ref{fig:twcube}.

\begin{figure}
\centering\begin{tikzpicture}
  \node[draw,circle] (a) at (0,2) {~};
  \node[draw,circle] (b) at (1.414,1.414) {~} edge (a);
  \node[draw,circle] (c) at (2,0) {~} edge (b);
  \node[draw,circle] (d) at (1.414,-1.414) {~} edge (c);
  \node[draw,circle] (e) at (0,-2) {~} edge (a) edge (d);
  \node[draw,circle] (f) at (-1.414,-1.414) {~} edge (b) edge (e);
  \node[draw,circle] (g) at (-2,0) {~} edge (c) edge (f);
  \node[draw,circle] (h) at (-1.414,1.414) {~} edge (a) edge (d) edge (g);
\end{tikzpicture}
\caption{The M\"{o}bius ladder $\Gamma_3$.}
\label{fig:twcube}
\end{figure}
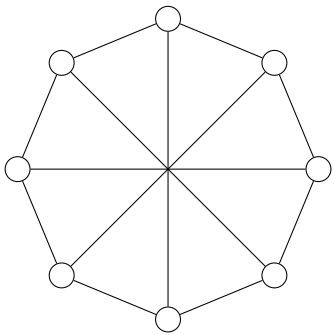

A few relevant pieces of notation from outside graph theory will be used.
Let $\Gamma(z)$ represent the usual gamma function (generalized factorial);
$n! = \Gamma(n+1)$ for integer $n$, but the gamma function is also
well-defined for arbitrary complex arguments.  We will use it only for
nonnegative reals, but not only for integers.  The similarity of notation
between $\Gamma(n+1)$ and $\Gamma_i$ is unfortunate, but these are 
widely-used standard symbols for these concepts. 
Some authors also use $\Gamma(v)$ for the neighbourhood of a vertex $v$; we
avoid that here.

For positive integers $n$ and $m$, let $I_n$ denote the $n\times n$ identity
matrix, and $J_{n,m}$ denote the $n\times m$ matrix with all entries equal to
$1$.  Given an $n \times n$ square matrix $A$, let $\perm{} A$ denote the
\emph{permanent} of $A$.  That is the sum, over all ways to choose $n$
entries from $A$ with one in each row and one in each column, of the product
of the chosen entries.  Note that the definition of the permanent is the
same as the definition of the determinant
without the alternating signs.


\subsection{Related work}

A number of previous results examine the problem of characterizing 
cycle-maximal graphs and bounding the number of cycles as a function of
girth, degree, or the number of edges for various classes of graphs.
Entringer and Slater \cite{Entringer:Maximum} show that 
some $n$-vertex graph with $m$ edges has at least $2^{m-n}$ cycles and 
every such graph has at most $2^{m-n+1}$ cycles.
Aldred and Thomassen \cite{Aldred:Maximum} improve the upper bound to
$(15/16)2^{m-n+1}$.
Guichard \cite{Guichard:Maximum} examines bounds on the number of cycles
to which any given edge can belong, including a discussion of cubic
graphs and triangle-free graphs.
Alt \emph{et al.}\ \cite{Alt:Number} show that the maximum number of 
cycles in any $n$-vertex planar graph 
is at least $2.27^n$ and at most $3.37^n$.
Buchin \emph{et al.}\ \cite{Buchin:Number} improve these bounds
to $2.4262^n$ and $2.8927^n$, respectively.
De Mier and Noy \cite{DeMier:Maximum}
examine the maximum number of cycles in outerplanar 
and series-parallel graphs. 
Knor \cite{Knor:Number} examines bounds on the maximum number of cycles 
in $k$-connected graphs,
including bounds expressed in terms of the minimum and maximum degrees.
Markstr\"om \cite{Markstrom:Extremal} presents results of a computer search 
examining the minimum and maximum numbers of cycles as a function 
of girth and the number of edges in small graphs.

Several results in extremal graph theory examine bounds on triangle-free
graphs.  Andr{\'a}sfai \emph{et al.}\ \cite{Andrasfai:Connection} show that every
$n$-vertex graph that has chromatic number $r$ but does not contain $K_r$ as
a subgraph has minimum degree at most $n (3r-7)/(3r-4)$.  Brandt
\cite{Brandt:Structure} examines the structure of triangle-free graphs with
minimum degree at least $n/3$.  Brandt and Thomass\'e \cite{Brandt:Dense}
show that every triangle-free graph with minimum degree greater than $n/3$
has chromatic number at most four.  Jin \cite{Jin:Chromatic} gives an upper
bound on the minimum degree of triangle-free graphs with chromatic number
four or greater.  Pach \cite{Pach:Graphs} characterizes triangle-free graphs
in which every independent set has a common neighbour: a triangle-free graph
has that property if and only if it is a maximal triangle-free graph
homomorphic to some $\Gamma_i$.  Brouwer \cite{Brouwer:Finite} provides a
simpler proof of Pach's result.


\section{Properties of triangle-free and cycle-maximal graphs}
\label{sec:props}


This section lists some properties of graphs
that we will use in subsequent sections.
Most of the proofs are simple, or already given by others, but we describe
them for completeness.

First, consider the gamma graphs defined in
Subsection~\ref{sub:definitions}.  This family of graphs recurs throughout
the literature on maximal triangle-free graphs.  They seem to have been
first introduced in 1964 by Andrásfai~\cite{Andrasfai:G18}.  Notation and
the order of labelling the vertices varies among authors; we follow Brandt
and Thomass\'e~\cite{Brandt:Dense} here.  All the $\Gamma_i$ graphs are
$i$-regular, circulant, and three-colourable.  As the following lemma
describes, the $\Gamma_i$ graphs form a hierarchy in which each one is
homomorphic to the next one, and deleting a vertex renders it homomorphic to
the previous one.

\begin{lemma}\label{lem:gamma-deleted}
For all $i>1$, $\Gamma_i$ with one vertex deleted is homomorphic to
$\Gamma_{i-1}$, and $\Gamma_{i-1}$ is homomorphic to $\Gamma_i$.
\end{lemma}

\begin{proof}
Let $v_1,v_2,\ldots,v_{3i-1}$ denote the vertices of $\Gamma_i$ and
$w_1,w_2,\ldots,w_{3i-4}$ denote the vertices of $\Gamma_{i-1}$.
Assume without loss of generality that $v_{3i-1}$ is the vertex deleted from
$\Gamma_i$.  Then define $f$ and $F$ as follows.
\begin{align*}
  f(v_j) &= \begin{cases}
    w_j & \text{if $j<i$}, \\
    w_{j-1} & \text{if $i\le j<2i$}, \\
    w_{j-2} & \text{if $j\ge 2i$};
  \end{cases} \\
  F(w_j) &= \begin{cases}
    v_j & \text{if $j<i$}, \\
    v_{j+1} & \text{if $i\le j <2i-2$}, \\
    v_{j+2} & \text{if $j \ge 2i-2$}.
  \end{cases}
\end{align*}
By checking their effects on the vertex neighbourhoods,
$f$ and $F$ are homomorphisms in both directions between $\Gamma_i$ with the
vertex $v_{3i-1}$ deleted, and $\Gamma_{i-1}$.  Reinserting the deleted
vertex, $\Gamma_{i-1}$ is also homomorphic to $\Gamma_i$.
\end{proof}

Several known results classify triangle-free graphs according to minimum
degree.  In particular, if a triangle-free graph $G$ has $n$ vertices and
minimum degree $\delta(G)$, then
\begin{itemize}
\item for every $i \in \{2,3,\ldots,10\}$, if $\delta(G)>in/(3i-1)$ then
  $G$ is homomorphic to $\Gamma_{i-1}$;
\item if $\delta(G)>2n/5$ then $G$ is bipartite;
\item if $\delta(G)>10n/29$ then $G$ is three-colourable; and
\item if $\delta(G)>n/3$ then $G$ is four-colourable.
\end{itemize}

Jin~\cite{Jin:Minimal} proves that $\delta(G)>in/(3i-1)$ implies $G$
homomorphic to $\Gamma_{i-1}$ for all $i$ up to $10$.  The case $i=2$, which
also implies $G$ is bipartite because $\Gamma_1=K_2$, was first proved by
Andr\'asfai~\cite{Andrasfai:G18}; a later paper, in English, by Andr\'asfai,
Erd\H{o}s, and S\'os, is often cited for this
result~\cite{Andrasfai:Connection}.  H\"{a}ggkvist proved the case
$i=3$~\cite{Haggkvist:Odd}.  Three-colourability when $\delta(G)>10n/29$
follows from the three-colourability of $\Gamma_9$.  Four-colourability when
$\delta(G)>n/3$ is due to Brandt and Thomass\'e~\cite{Brandt:Dense}.

The following property of cycle-maximal graphs applies to graphs of general
girth, not only triangle-free graphs: we can limit consideration to
2-connected graphs.

\begin{lemma}\label{lem:maybe-biconn}
Let  $3\leq g\leq n$.  Among all $n$-vertex cycle-maximal graphs for
girth at least $g$, there is one that is 2-connected.
\end{lemma}

\begin{proof}
Because $g\le n$, there exists a graph with one cycle and these parameters. 
That graph consists of a
cycle of length $g$ and $n-g$ degree-zero vertices.  Therefore any
cycle-maximal graph for girth at least $g$ contains at least one cycle.

Given a disconnected graph with maximal cycle count, choose a vertex $v$;
then choose one vertex in each connected component other than the one
containing $v$, and add an edge from each of those vertices to $v$.  The
resulting connected graph contains all and only the cycles from the
original, so it has the same girth and cycle count.  Therefore we need only
consider connected graphs.

Any block either is a single edge, or contains a
cycle; if it is a single edge, it cannot be part of any cycle.  We can
contract it without removing any cycles nor decreasing the girth, and then
insert one new vertex to replace the one we eliminated, in the middle of
some edge that is part of a cycle.  Therefore we need only consider
blocks that contain cycles, necessarily of at least $g$
vertices.

Suppose there is a cut-vertex $u$.  Removing it would disconnect at least
two blocks; let $v$ and $w$ be two vertices maximally
distant from $u$ that would be disconnected from each other
by the removal of $u$.  Each of $v$ and $w$
is at least distance $\lfloor g/2 \rfloor$ from $u$.  Then by adding an
edge $(v,w)$, we create at least one new cycle, but none of length less
than $g$.
\end{proof}

In the case of triangle-free graphs, Lemma~\ref{lem:maybe-biconn} can be
strengthened to require 2-connectedness in all cycle-maximal graphs.

\begin{corollary}\label{cor:definitely-biconn}
All cycle-maximal triangle-free graphs with at least four vertices are
2-connected.
\end{corollary}

\begin{proof}
Suppose $G$ is a cycle-maximal triangle-free graph with at least four
vertices.  Because $C_4$ contains a cycle, $G$ contains at least one cycle
and therefore at least one vertex of degree at least two.  If $G$ is
disconnected, let $u$ and $v$ be two vertices in distinct components and
with the degree of $u$ at least two.  Then add edges from $v$ to all
neighbours of $u$.  These edges do not create any triangles, but create at
least one new cycle through $u$, $v$, and two neighbours of $u$,
contradicting the cycle-maximality of $G$.  Therefore $G$ is connected.

Suppose $G$ contains a block that is a single edge.  Then as in the proof of
Lemma~\ref{lem:maybe-biconn} we can contract it, removing a vertex while keeping
all cycles and not creating any triangles; and then we can add a new vertex
$v$ sharing all the neighbours of some vertex $u$ with degree at least two. 
By doing so we create at least one new cycle through $u$, $v$, and two
neighbours of $u$, contradicting the cycle-maximality of $G$.  Therefore $G$
contains no blocks that are single edges.  All remaining cases are covered
by the last paragraph of the proof of Lemma~\ref{lem:maybe-biconn}.
\end{proof}

The next property is also specific to the triangle-free case: every
edge in a cycle-maximal graph is part of some minimum-length cycle.

\begin{lemma}\label{lem:every-edge-cfour}
If $G$ is a cycle-maximal triangle-free graph with at least four vertices,
then $G$ is maximal triangle-free and every edge in $G$ is in some 4-cycle.
\end{lemma}

\begin{proof}
Suppose $u$ and $v$ are non-adjacent vertices in $G$ and adding the edge
$(u,v)$ would not create a triangle.  By 2-connectedness
(Corollary~\ref{cor:definitely-biconn}) there exist
two edge-disjoint paths from $u$ to $v$ in $G$, and then adding the edge
$(u,v)$ creates at least two new cycles, contradicting cycle-maximality;
therefore $G$ is maximal triangle-free.

Suppose $(u,v)$ is an edge in $G$ that is not part of any 4-cycle.  Let $G'$
be the graph formed from $G$ by contracting $(u,v)$.  This operation cannot
create any triangles; and $G'$ contains one less vertex than $G$ and all the
cycles of $G$ except any that included both $u$ and $v$ without including
the edge $(u,v)$.  Let $w$ be the vertex created by the edge contraction; and
add a new vertex $w'$ to $G'$ with the same neighbourhood as $w$.  For each
cycle in $G$ that used $u$ and $v$ without the edge between them, the new
graph contains at least one cycle using $w$ and $w'$ instead; and there is
also at least one new 4-cycle through $w$, $w'$, and two of their
neighbours.  (They have at least two neighbours because $G$ was
2-connected.)  Therefore we have increased the number of cycles for an
$n$-vertex triangle-free graph, contradicting cycle-maximality.  Therefore
every edge in $G$ is part of some 4-cycle.
\end{proof}

Also note that by a result of Erd\H{o}s \emph{et
al.}~\cite[Lemma~2.4(ii)]{Erdos:How}, any triangle-free graph (not only
maximal or cycle-maximal) with $n$ vertices and $m$ edges has at least one
edge contained in at least $4m(2m^2-n^3)/n^2(n^2-2m)$ cycles of length four.

Lemma~\ref{lem:maybe-biconn} and Corollary~\ref{cor:definitely-biconn} do
not generalize to higher girth.  A graph consisting of the Petersen graph
plus one vertex added with degree one is cycle-maximal for 11 vertices and
girth at least five, but the edge to the added vertex is not in any 5-cycle,
nor any cycle at all, and the graph is not 2-connected.

Finally, we list some simple equivalent conditions for cycle-maximal
triangle-free graphs to be the Tur\'an graph.  Any counterexample to
Conjecture~\ref{con:main} would have to lack all these properties.

\begin{lemma}\label{lem:equivalent-props}
If $G$ is a cycle-maximal triangle-free graph with $n\ge 4$ vertices,
then these statements are equivalent:
\begin{enumerate}
\item\label{itm:prop-turan} $G$ is the bipartite Tur\'an graph $T(n,2)$;
\item\label{itm:prop-complete} $G$ is complete bipartite;
\item\label{itm:prop-bipartite} $G$ is bipartite;
\item\label{itm:prop-perfect} $G$ is perfect;
\item\label{itm:prop-no-pfour} $G$ contains no induced $P_4$; and
\item\label{itm:prop-min-degree} for $n\ne 5$, $G$ has minimum degree
  greater than $2n/5$.
\end{enumerate}
\end{lemma}

\begin{proof}
The bipartite Tur\'an graph has all the listed properties
($\ref{itm:prop-turan}\Rightarrow \{
\ref{itm:prop-complete},\ref{itm:prop-bipartite},\ref{itm:prop-perfect},
\ref{itm:prop-no-pfour},\ref{itm:prop-min-degree}\}$),
so it remains to
prove the implications in the other direction.
By exact cycle count, $T(n,2)$ maximizes cycles among complete bipartite
graphs (see Corollary~\ref{cor:complete-bipartite};
$\ref{itm:prop-complete}\Rightarrow\ref{itm:prop-turan}$).
If $G$ is bipartite, it is necessarily
complete bipartite in order to be maximal triangle-free
($\ref{itm:prop-bipartite}\Rightarrow\ref{itm:prop-complete}$). 
Triangle-free perfect graphs are bipartite as a trivial consequence of the
definition ($\ref{itm:prop-perfect}\Rightarrow\ref{itm:prop-bipartite}$).
Any graph without an induced $P_4$ is perfect by a result of
Seinsche~\cite{Seinsche:Property}, with a simpler
proof given by Arditti and de~Werra~\cite{Arditti:Note}
($\ref{itm:prop-no-pfour}\Rightarrow\ref{itm:prop-perfect}$).
Any triangle-free graph with minimum degree greater than $2n/5$ is bipartite
($\ref{itm:prop-min-degree}\Rightarrow\ref{itm:prop-bipartite}$)~\cite{Andrasfai:G18,Andrasfai:Connection}.
\end{proof}

Our Theorems~\ref{thm:regular-triangle-free}
and~\ref{thm:near-reg-triangle-free} have the effect of adding ``$G$ is
regular'' to the list of equivalent conditions for
all even $n$, and ``$G$ is near-regular'' for odd $n>91$.


\section{Bounds on cycle counts}
\label{sec:bounds}


In this section we prove bounds on the numbers of cycles in certain kinds of
graphs.  We have three basic kinds of bounds, each of which admits some
variations.  First, for the bipartite Tur\'{a}n graph $T(n,2)$ it is
possible to compute the number of cycles exactly for any given $n$, but the
resulting expression is a summation; we also find a reasonably tight
closed-form lower bound.  We can then rule out potential counterexamples to
Conjecture~\ref{con:main} by showing upper bounds on the number of cycles in
other kinds of graphs.  The remaining two kinds of bounds are based on the
number of edges, and on homomorphism.

The asymptotic results come from applying Stirling's approximation for the
factorial
in the following form, which gives precise upper and lower bounds.  Note
that the approximation is actually an approximation for the gamma function,
so we can apply it to non-integer arguments.  The approximation is:

\begin{gather}
  n \ln n - n + \frac{1}{2}\ln n + \frac{1}{2}\ln 2\pi
  \le \ln \Gamma(n+1) \label{eqn:stirling-lb} \, , \text{ and} \\
  \ln \Gamma(n+1) \le
  n \ln n - n + \frac{1}{2}\ln n + \frac{1}{2}\ln 2\pi + \frac{1}{12}\cdot
  \frac{1}{n} \, . \label{eqn:stirling-ub}
\end{gather}

Our general approach will be to prove bounds on $\ln c(G)$ as a function of
$n$ for $G$ in various classes of graphs.  The bounds typically take the
form $n \ln n - cn + O(\ln n)$ for some constant coefficient $c \ge 1$. 
These amount to proofs that the number of cycles is on the order of $n!$
divided by some exponential function, with the coefficient of $n$ in $\ln
c(G)$ describing the size of the exponential function.  Comparing the
coefficients suffices to show that one class of graphs has more cycles than
another for sufficiently large $n$; and with more careful attention to the
lower-order terms we can bound the values of $n$ that are ``sufficiently
large,'' leaving a known finite number of smaller cases to address with
other techniques.


\subsection{Cycles in $T(n,2)$}

It is relatively easy to count cycles in the
bipartite Tur\'an graph $T(n,2)$.  The following result gives the exact
count as a summation, and an asymptotic approximation.

\begin{lemma}
The number of cycles in $T(n,2)$ is given exactly by
\begin{equation}
  c(T(n,2)) =
  \sum_{k=2}^{\lfloor n/2 \rfloor} \frac{\lfloor n/2 \rfloor!\lceil n/2
  \rceil!}{2k(\lfloor n/2 \rfloor-k)!(\lceil n/2 \rceil-k)!} \, ,
  \label{eqn:ktwo-exact}
\end{equation}
and satisfies the bound
\begin{equation}
\begin{aligned}
  \ln c(T(n,2)) &\ge n \ln n -(1+\ln 2)n + \ln \pi \\
  &\approx n \ln n - 1.693147n + 1.44730 \, .
\end{aligned}
\label{eqn:ktwo-numerical}
\end{equation}
\end{lemma}

\begin{proof}
To describe a cycle in the bipartite Tur\'{a}n graph $T(n,2)$, we can
start by choosing a value $k$ to be the number of vertices the cycle
includes on each side of the bipartite graph.  The length of the cycle will
be $2k$, and necessarily $2 \le k \le \lfloor n/2 \rfloor$.  Then we choose
a permutation for $k$ of the $\lfloor n/2 \rfloor$ vertices in the smaller
part, and a permutation for $k$ of the $\lceil n/2 \rceil$ vertices in the
larger part.  These choices will describe each possible cycle $2k$ times,
because there are $k$ equivalent starting points and two equivalent
directions.  Therefore we divide by $2k$ to avoid overcounting, and the
overall total number of cycles is given by \eqref{eqn:ktwo-exact}.

The term for $k=\lfloor n/2 \rfloor$ is by far the largest, so we can use it
alone as a reasonably tight lower bound.  The factorials in the denominator
become one and drop out.  By the properties of the
gamma function, $\lfloor n/2 \rfloor!\lceil n/2 \rceil! \ge
\Gamma((n/2)+1)^2$,
so we can drop the floors and ceilings in the numerator,
use gamma instead of factorial, and have a valid lower bound for both even
and odd $n$.  Similarly, replacing
$2\lfloor n/2\rfloor$ by $n$ in the denominator does not increase the bound.
We have:
\begin{equation*}
  c(T(n,2)) \ge \frac{\Gamma((n/2)+1)^2}{n} \, .
\end{equation*}

Applying Stirling's approximation \eqref{eqn:stirling-lb}
gives \eqref{eqn:ktwo-numerical}.
\end{proof}

The following corollary confirms the intuition that $T(n,2)$ should
have more cycles than a less-balanced complete bipartite graph.

\begin{corollary}\label{cor:complete-bipartite}
The graph $T(n,2)$ for $n\ge 4$ is uniquely
cycle-maximal among complete bipartite graphs on $n$ vertices.
\end{corollary}

\begin{proof}
The requirement $n\ge 4$ is to rule out pathological cases in which no
cycles are possible at all.  Let $a$ and $b$ represent the sizes of the two
parts, with $n=a+b$ and assume without loss of generality $a\le b$.  The
number of cycles in $K_{a,b}$ is a suitably modified version of
\eqref{eqn:ktwo-exact}:
\begin{align*}
  c(K_{a,b})
    &= \sum_{k=2}^{a} \frac{a!b!}{2k(a-k)!(b-k)!} \\
   &= \sum_{k=2}^{a} \frac{1}{2k}
    \cdot (ab)
    \cdot \left( (a-1)\cdot(b-1) \right)
    \cdots
    \left( (a-k+1)\cdot(b-k+1) \right) \, .
\end{align*}

If $b>a+1$, then subtracting one from $b$ and adding one to $a$ will
strictly increase all the factors $(ab)$,
$\left( (a-1)\cdot(b-1)\right)$, and so on.
Making this change will also add an additional positive term to the sum.
Therefore the sum is uniquely
maximized when $b\le a+1$, which means the graph is $T(n,2)$.
\end{proof}


\subsection{Cycles as a function of number of edges}

It seems intuitively reasonable that more edges should mean more cycles. 
We can make that more precise by giving an upper bound on number
of cycles as a function of number of edges, and therefore (by comparison
with the previous bound) a lower bound on number of edges necessary for a
graph to potentially exceed the number of cycles in the bipartite Tur\'{a}n
graph.  First, we define notation for the maximal product of a constrained
sequence of integers, which will be used in bounding the cycle count.

\begin{definition}\label{def:pi-func}
Let $\Pi(n,m)$, with $2\le m \le \binom{n}{2}$,
denote the greatest possible product for any $k<
n$ of a sequence of positive integers $c_1,c_2,\ldots,c_{k}$ with
$c_i\le n-i$ for all $1 \le i\le k$ and $\sum_{i=1}^{k} c_i = m$.
\end{definition}

The following lemma describes the value of $\Pi(n,m)$.

\begin{lemma}\label{lem:pifunc-value}
If $m=\binom{n}{2}$, then
\begin{equation}
  \Pi(n,m) = (n-1)! \,. \label{eqn:pi-factorial}
\end{equation}
If $2\le m \le 3n-7$, then
\begin{equation}
  \Pi(n,m) =
    \begin{cases}
      3^{m/3}
        & \text{ for } m \equiv 0 \pmod{3}; \\
      4\cdot 3^{(m-4)/3}
        & \text{ for } m \equiv 1 \pmod{3}; \text{ and} \\
      2\cdot 3^{(m-2)/3}
        & \text{ for } m\equiv 2 \pmod{3}\, .
    \end{cases} \label{eqn:pi-unconstrained}
\end{equation}
If $3n-7<m<\binom{n}{2}$, then $k=n-2$ and there exist integers
$s\ge 3$ and $t\ge 0$ such that
\begin{equation}
  \Pi(n,m) = (s+1)^t s^{n-s-t} (s-1)! \, . \label{eqn:pi-constrained}
\end{equation}
\end{lemma}

\begin{proof}
In the case $m=\binom{n}{2}$, the only sequence satisfying the constraints
is $n-1, n-2, \ldots, 1$ and $\Pi$ is the product of that sequence, giving
\eqref{eqn:pi-factorial}.

Sorting the $c_i$ into nonincreasing order cannot cause them to violate the
constraints, so we assume it.  Removing a $c_i$ term greater than $3$ and
replacing it with two terms $c_i-2$ and $2$ will never decrease the product. 
Removing a term equal to $1$ and adding $1$ to some other term will always
increase the product, as will removing three terms equal to $2$ and
replacing them with two terms equal to $3$.  Repeated
application of these rules uniquely determines a sequence ending with at
most two terms equal to $2$, all other terms equal to $3$, and if
the constraints allow this sequence, then it determines $\Pi$, giving
\eqref{eqn:pi-unconstrained}.

Subtracting $1$ from a term $c_i>3$ and adding $1$ to some other term less
than $c_i-1$ will always increase the product.  Repeated application of that
operation and the operations used for \eqref{eqn:pi-unconstrained}, wherever
permitted by the constraints, uniquely determines a sequence in the form
given by \eqref{eqn:pi-constrained}.
\end{proof}

Now the $\Pi$ function is applied to bound the number of cycles.

\begin{lemma}\label{lem:edge-bound}
If a graph $G$ has $n$ vertices, $m$ edges,
and girth at least $g$, then
\begin{equation}
  c(G) \le \Pi(n-1,m) \frac{n^2}{2g} \label{eqn:edges-pifunc}
    \, ,
\end{equation}
and if $3n-7<m<\binom{n}{2}$,
\begin{equation}
  \ln c(G) \le
    n \ln n
    - (\alpha - \ln \alpha) n
    + \frac{5}{2} \ln n
    + \frac{1}{2} \ln \alpha 
    + \frac{1}{2}\ln \frac{\pi}{2}
    - \ln g
    + \frac{1}{12\alpha n}
    \label{eqn:edges-stirling}
    \, ,
\end{equation}
where $\alpha = 1-\sqrt{1-1/n-2(m+1)/n^2}$.
\end{lemma}

\begin{proof}
Suppose we are counting Hamiltonian cycles in a complete graph.  We
might start at the first vertex, leave via one of its $n-1$ edges, then from
the next vertex, choose one of the $n-2$ edges remaining (excluding the one
from the first vertex), and so on.  At the last vertex, there are no
remaining edges to previously unvisited vertices, and we return to the
starting point.  Overall there are $(n-1)!$ choices of successor vertices,
which suffices as an upper bound.  Note that $(n-1)!$ is the product of
$n-1$ positive integer factors whose sum is exactly the number of edges in
the complete graph.  Every time we consider an edge as a choice for leaving
a vertex, that edge is eliminated from consideration for all future
vertices, hence the bound on the sum.  The last few factors in the sequence
are $3,2,1$ because we can only visit a previously unvisited vertex
and no term can be greater than the number of previously unvisited vertices
that remain.

For a more general graph $G$ with $n$ vertices, $m$ edges, girth at
least $g$, and cycles that might not be maximal length, we can follow a
similar procedure.  There are at most $n-1$ positive integers representing
choices of successors of all but the last vertex; their sum is at most $m$;
and if the factors are $c_1,c_2,\ldots,c_{k}$, the remaining vertex
constraint is $c_i\le n-i$ for all $1 \le i \le k$.  By
definition, $\Pi(n,m)$ is an upper bound on the product of such a
sequence.

Since we are not requiring cycles to be Hamiltonian, we cannot assume that
any single vertex is the first one in the cycle or is in the cycle at all,
so we multiply the bound by $n$ to account for choosing any starting vertex. 
To account for choosing the length, we multiply by $n$ for
choosing which vertex is the last vertex, assume that the cycle closes as
soon as it reaches that vertex, and then any remaining choices we may have
counted for vertices not in the cycle will only go to make the upper bound a
little less tight.  Finally, we can remove a small amount of overcounting. 
With a girth of $g$ (necessarily at least $3$) there will be $g$ distinct
choices of starting vertex that actually generate the same cycle; and we can
always generate each cycle in two equivalent directions.  So we can divide
by $2g$ and still have a valid upper bound.  Multiplying $\Pi(n,m)$ for
choices of successors with $n^2$ for choices of starting and ending
vertices, and dividing by $2g$, gives exactly the bound
\eqref{eqn:edges-pifunc}.

The form of the sequence $c_i$ that achieves $\Pi(n,m)$ is described in
Lemma~\ref{lem:pifunc-value}.  In the case $3n-7<m<\binom{n}{2}$ (dense but
not complete graphs), this sequence is of length $n-2$ and in general is
of the form $\lceil \alpha n \rceil, \ldots, \lceil \alpha n \rceil,
\lfloor \alpha n \rfloor, \ldots, \lfloor \alpha n \rfloor,
\lfloor \alpha n \rfloor -1, \lfloor \alpha n \rfloor-2, \ldots, 4,3,2$, for
some $\alpha$ chosen so that the sum of the sequence is $m$. 
Note that there is no final factor of $1$ counted in the sequence, because
adding it to an earlier term gives a greater product.
These factors are shown schematically in Figure~\ref{fig:cg-factors}.

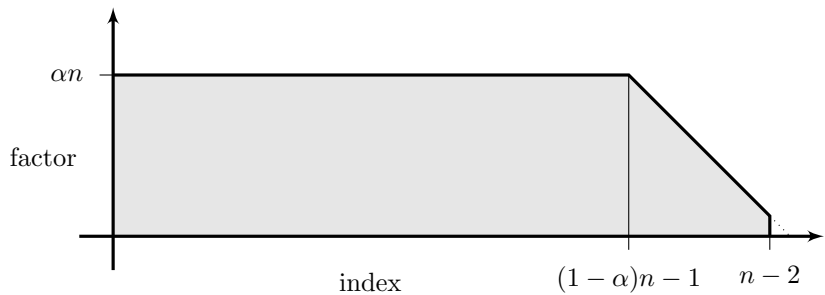
\begin{figure}
  \begin{tikzpicture}[>=latex',scale=0.9]
    \fill[color=black!10!white] (0,0)--(0,2.384)--(7.616,2.384)
      --(9.7,0.3)--(9.7,0)--cycle;
    \draw (-0.2,2.384)--(0,2.384);
    \draw (7.616,-0.2)--(7.616,2.384);
    \draw (9.7,-0.2)--(9.7,0);
    \node[anchor=east] at (-0.3,2.384) {$\alpha n$};
    \node[anchor=east] at (-0.4,1.172) {factor};
    \node[anchor=north] at (7.616,-0.3) {$(1-\alpha)n-1$};
    \node[anchor=north] at (9.7,-0.3) {$n-2$};
    \node[anchor=north] at (3.8,-0.4) {index};
    \draw[dotted] (9.7,0.3)--(10,0);
    \draw[very thick] (0,2.384)--(7.616,2.384)--(9.7,0.3)--(9.7,0);
    \draw[very thick,->] (-0.5,0)--(10.5,0);
    \draw[very thick,->] (0,-0.5)--(0,3.384);
  \end{tikzpicture}
  \caption{Factors in the upper bound on $\Pi(n,m)$.}
  \label{fig:cg-factors}
\end{figure}

If $\alpha n$ is an integer, then this product is $(\alpha
n)^{(1-\alpha)n}(\alpha n -1)!$.  
The sum is $(\alpha n)(1-\alpha n)n+(\alpha n -1 + \alpha n -2 + \cdots +
3 + 2)$.  Setting that to $m$ and applying the usual formula for the sum of
consecutive integers, we have $(-n^2/2)\alpha^2 + (n^2-n/2)\alpha -m-1=0$. 
Solving the quadratic, choosing the solution between $0$ and $1$, and
removing some terms for an upper bound, gives
\begin{align*}
  \alpha &= 1-\frac{1}{2n}
      -\sqrt{1-\frac{1}{n}+\frac{1}{4n^2}-\frac{2(m+1)}{n^2}} \\
    &\le 1-\sqrt{1-\frac{1}{n}-\frac{2(m+1)}{n^2}} \, .
\end{align*}

For $\alpha n$ not an integer,
removing the floors and ceilings outside the factorial can only increase the
product, because making those terms equal maximizes their product given that
their sum is fixed.
The factorial is at most $\lceil \alpha n -1 \rceil!$,
and changing it to
$\Gamma(\lceil \alpha n -1 \rceil + 1) \le \Gamma(\alpha n +1)$
similarly cannot decrease the product.  Where
$\alpha=1-\sqrt{1-1/n-2(m+1)/n^2}$,
we have
\begin{equation*}
  \Pi(n,m) \le
    (\alpha n)^{(1-\alpha)n} \Gamma(\alpha n + 1) \, .
\end{equation*}
The result \eqref{eqn:edges-stirling} follows by Stirling's
approximation~\eqref{eqn:stirling-ub}:
\begin{align*}
\ln c(G) &\le (1-\alpha)n \ln \alpha n + \ln \Gamma(\alpha n+1)
    + \ln \frac{n^2}{2g} \\
  &\le n \ln n
    - (\alpha - \ln \alpha) n
    + \frac{5}{2} \ln n
    + \frac{1}{2} \ln \alpha 
    + \frac{1}{2}\ln \frac{\pi}{2}
    - \ln g
    + \frac{1}{12\alpha n} \, . \qedhere
\end{align*}
\end{proof}

A cycle-maximal triangle-free graph $G$ necessarily contains
enough edges for \eqref{eqn:edges-stirling} to exceed
\eqref{eqn:ktwo-numerical}.  For sufficiently large $n$, the coefficients of
$n$ in the bounds on $\ln c(G)$
will determine which bound is greater; for \eqref{eqn:edges-stirling} to
exceed \eqref{eqn:ktwo-numerical} requires that
$\alpha-\ln \alpha\le 1+\ln 2$. 
Then $\alpha \ge 0.231961\ldots$ and $2m/n$
(the average degree of $G$) is at least
$n(0.410116\ldots)$.  Critically, that is greater than $2n/5$.  In a
graph that is regular, or close to regular in the sense that the difference
between minimum and maximum degrees is bounded by some constant, the
minimum degree approaches the average and so
is also greater than $2n/5$ for sufficiently large $n$.
But any triangle-free graph with minimum degree greater than $2n/5$
is bipartite~\cite{Andrasfai:G18,Andrasfai:Connection},
giving the following corollary.

\begin{corollary}\label{cor:near-regular}
Let $k$ be any fixed nonnegative integer and let $G$ be any cycle-maximal
triangle-free graph with $n$ vertices and $\Delta(G)-\delta(G)\le k$.
Then for sufficiently large $n$, $G$ is the bipartite Tur\'{a}n graph.
\end{corollary}

In particular, note that $C_5(t)$, which is an important case for many
previous results on maximal triangle-free graphs including that of
Andr{\'a}sfai used above~\cite{Andrasfai:G18,Andrasfai:Connection}, is
regular with degree exactly $2n/5$ and so is \emph{not} cycle-maximal
triangle-free once $n$ is sufficiently large.  It does not have enough edges
to be cycle-maximal triangle-free.  Neither does any other non-bipartite
regular graph for sufficiently large $n$.  Later in the present work, when
we show that no regular graph is a counterexample to
Conjecture~\ref{con:main}, we need only consider the finite number of cases
in which $n$ is not ``sufficiently large.''

However, this result concerns average degree, not minimum degree.  A graph
could have a large gap between average and minimum degrees.  For instance,
the graph formed by inserting a degree-two vertex in one edge of $T(n-1,2)$
has average degree approaching $n/2$ despite its minimum degree being fixed
at $2$; and although it clearly has fewer cycles than $T(n,2)$,
Lemma~\ref{lem:edge-bound} is not strong enough to prove that.  Note that by
a result of Erd\H{o}s~\cite[Lemma~1]{Erdos:Rademacher}, this graph also
contains the maximum possible number of edges for a non-bipartite
triangle-free graph on $n$ vertices.


\subsection{Cycle bounds from homomorphisms}

Several important results on maximal triangle-free graphs amount to proving
that a graph $G$ with certain properties is necessarily homomorphic to some
fixed, usually small, graph $H$.  The following lemmas provide bounds on
the number of cycles in a graph with that kind of homomorphism; first for
$G$ a uniform blowup of $H$, and then more generally where the sizes of the
sets mapping onto each vertex of $H$ are known but not necessarily all the
same.

\begin{lemma}\label{lem:hmorph-bound}
If $G$ and $H$ are graphs with $n$ and $p$ vertices respectively, $n$
an integer multiple of $p$, $G$ is a subgraph of $H(n/p)$, $g$ is the girth
of $G$, and $q=\Delta(H)$, then
\begin{gather}
  c(G) \le q^n \left[ \left( \frac{n}{p} \right) ! \right]^p
            \frac{n}{2g} \label{eqn:hmorph-qn}
    \, , \text{ and} \\
  \ln c(G) \le
    n \ln n
    - \left( 1 + \ln \frac{p}{q} \right) n
    + \left( 1 + \frac{p}{2}\right) \ln n
    + \frac{p}{2}\ln \frac{2\pi}{p}
    - \ln 2g
    + \frac{p^2}{12n}
    \, .
    \label{eqn:hmorph-stirling}
\end{gather}
\end{lemma}

\begin{proof}
For each vertex in $G$, we will choose a successor in $H$.  There are at
most $q^n$ ways to do that.  By also choosing a permutation for
the $n/p$ vertices in $G$ corresponding to each of the $p$ vertices of $H$
(overall $(n/p)!^p$ choices), we can uniquely determine a successor for each
vertex in the cycle.
Note that we can choose any arbitrary
successors for vertices not in the cycle, since we have not limited the
total number of times we might choose a vertex of $H$; special handling of
non-cycle vertices as in Lemma~\ref{lem:perm-bound} is not necessary here.

The starting vertex is determined by choosing one of the $p$ partitions.  To
determine the length of the cycle, bearing in mind that the cycle can only
end when it returns to its initial partition, we can choose how many of the
$n/p$ vertices in the initial partition to include in the cycle. 
Multiplying those factors, the $p$ cancels out, leaving a factor of $n$ for
the choice of both starting vertex and cycle length.  Alternately, this
choice can be viewed as selecting from among $n$ vertices one to be the
\emph{last} vertex in the cycle, with the starting partition implicitly the
partition containing that vertex, and the starting vertex implicitly the
first one in the starting partition according to the earlier-counted vertex
permutations.  We can also remove a factor of $2g$ because any cycle
(necessarily of length at least $g$) can be described using any of $2$
directions and at least $g$ starting vertices.  Multiplying all these
factors gives \eqref{eqn:hmorph-qn}.

Then
\eqref{eqn:hmorph-stirling} follows by Stirling's approximation as follows:
\begin{align*}
\ln c(G) &\le n \ln q + p\left[ \frac{n}{p} \ln \frac{n}{p} - \frac{n}{p}
    +\frac{1}{2}\ln \frac{n}{p} + \frac{1}{2} \ln 2\pi +\frac{p}{12n}\right]
    + \ln \frac{n}{2g} \\
  &= n \ln q + n \ln \frac{n}{p} - n + \frac{p}{2} \ln \frac{n}{p}
    +\frac{p}{2} \ln 2\pi + \frac{p^2}{12n} + \ln \frac{n}{2g} \\
  &= n \ln n
    - \left( 1 + \ln \frac{p}{q} \right) n
    + \left( 1 + \frac{p}{2}\right) \ln n
    + \frac{p}{2}\ln \frac{2\pi}{p}
    - \ln 2g
    + \frac{p^2}{12n}
    \, . \qedhere
\end{align*}
\end{proof}


Lemma~\ref{lem:hmorph-bound} can potentially overcount by a significant
margin because of the $q^n$ term, which allows each vertex of $G$ to choose
a successor in $H$ without restriction.  A Hamiltonian cycle in $G$ would
necessarily visit each vertex of $H$ exactly $n/p$ times, not any arbitrary
number of times; many of the $q^n$ successor-in-$H$ choices involve choosing
a vertex of $H$ more than $n/p$ times and so cannot actually correspond to
feasible full-length cycles in $G$.  There are many fewer than $q^n$ ways to
choose each vertex of $H$ exactly $n/p$ times while obeying the other
applicable constraints.  The situation is complicated somewhat by the
possibility of non-Hamiltonian cycles, but
it remains that the bound \eqref{eqn:hmorph-qn} is quite loose for many
graphs of interest.

The matrix permanent offers a way to prove a tighter upper bound on cycles
given a homomorphism.
The following result replaces the successor choice in
Lemma~\ref{lem:hmorph-bound} with a computation of the permanent of the
adjacency matrix of the graph.  Choosing a permutation of the rows and
columns for which all the chosen entries of the adjacency matrix are nonzero
corresponds to choosing a neighbour as successor for each vertex in the
graph such that each vertex is chosen exactly once, and the permanent counts
such choices, including all Hamiltonian cycles.  To allow for
non-Hamiltonian cycles, which might not involve all vertices, we add loops
to all the vertices, corresponding to ones along the diagonal of the matrix,
allowing any vertex to choose itself as successor and therefore not need to
be chosen by any other vertex.  The result is a simple upper bound on number
of cycles.  This approach is also more easily applicable to non-uniform
blowups; that is, where different vertices in $H$ do not all correspond to
the same size of independent sets in $G$.  We will discuss later how to
compute the permanent efficiently for the cases of interest here.

\begin{lemma}\label{lem:perm-bound}
In a graph $G$ with $n$ vertices whose adjacency matrix is $(g_{ij})$,
\begin{equation}
  c(G) \le \frac{1}{2}\perm \left( (g_{ij}) + I_n \right ) \, .
  \label{eqn:perm-bound}
\end{equation}
Furthermore, if $G$ is homomorphic to a
graph $H$ with $p$ vertices labelled $1\ldots p$ and adjacency matrix
$(h_{ij})$, via
a homomorphism $f:V(G) \rightarrow V(H)$ that maps $n_i=|f^{-1}(i)|$
vertices of $G$ to each vertex $i$ of $H$, then
\begin{equation}
  c(G) \le \frac{1}{2}\perm
    \begin{pmatrix}
      I_{n_1} & h_{12}J_{n_1n_2}          & \ldots & h_{1p}J_{n_1n_p} \\
      h_{21}J_{n_2n_1} & I_{n_2}          & \ldots & h_{2p}J_{n_2n_p} \\
      \vdots & \vdots &                     \ddots & \vdots \\
      h_{p1}J_{n_pn_1} & h_{p2}J_{n_pn_2} & \ldots & I_{n_p}
    \end{pmatrix} \label{eqn:perm-ibound} \, .
\end{equation}

\end{lemma}
\begin{proof}
A directed cycle cover, or oriented 2-factor, of $G$ is a choice, for each
vertex $v$ in $G$, of a successor vertex adjacent to $v$ such that each
vertex is chosen exactly once.  If we add a loop to every vertex of $G$
(making each vertex adjacent to itself), then every cycle in $G$ is uniquely
determined by at least two directed cycle covers of the resulting graph:
namely those in which the cycle vertices choose their successors in the
cycle, going around the cycle in either direction, and any other vertices
choose themselves.  The permanent of $\left( (g_{ij}) + I_n \right )$ counts
exactly those directed cycle covers, and dividing it by two for the two
directions gives \eqref{eqn:perm-bound}.

When $G$ is homomorphic to $H$, we can assume for an upper bound that
$G$ contains all edges allowed
by the homomorphism; adding edges does not decrease the number of cycles.
Then
\eqref{eqn:perm-ibound} is just \eqref{eqn:perm-bound} applied to the
maximal graph.
\end{proof}


\section{Cycles in regular triangle-free graphs}
\label{sec:regular}

By Corollary~\ref{cor:near-regular}, no regular graph with $n$ vertices
except $T(n,2)$ can be cycle-maximal triangle-free for $n$ sufficiently
large.  In this section we show that in fact that statement applies to all
$n$.

Recall that a maximal
triangle-free graph with $n$ vertices and minimum degree greater than
$10n/29$ is homomorphic to some $\Gamma_i$ with $i\le 9$.  If the graph
is also regular, the following lemma narrows the possibilities further.

\begin{lemma}\label{lem:gamma-blowup}
An $n$-vertex regular maximal triangle-free graph $G$ homomorphic to some
$\Gamma_i$ is exactly $\Gamma_j(n/(3j-1))$ for some $j \le i$.
\end{lemma}

\begin{proof}
Edge-maximality implies $G$ is exactly $\Gamma_i$ with all vertices replaced
by independent sets and all the edges that are allowed by the homomorphism;
that is, $\Gamma_i(n_1,n_2,\ldots,n_p)$ with $p=3i-1$.
Suppose one of those independent sets is empty; then some $n_k=0$ and $G$ is
homomorphic to $\Gamma_i$ minus one vertex.
But by Lemma~\ref{lem:gamma-deleted}, deleting a vertex from $\Gamma_i$
leaves a graph homomorphic to $\Gamma_{i-1}$.  By transitivity $G$ is
homomorphic to $\Gamma_{i-1}$, and by induction there exists $j\le i$
such that $G=\Gamma_j(n_1,n_2,\ldots,n_{3j-1})$ with all the $n_k>0$.

The neighbourhoods of $v_{2j}$ and $v_{2j+1}$ in $\Gamma_j$ are
$\{v_1,v_2,\ldots,v_j\}$ and $\{v_2,v_3,$ $\ldots,$
$v_{j+1}\}$ respectively;
these differ only by the substitution of $v_{j+1}$ for $v_1$.  If $G$ is
regular, we have
\begin{equation*}
  \sum_{k=1}^{j} n_k = \sum_{k=2}^{j+1} n_k \quad\text{and}\quad
  n_1 = n_{j+1} \, .
\end{equation*}
Symmetrically around the cycle, $n_k=n_{j+k}$ for all $k$, taking the
subscripts modulo $3j-1$.  Because $j$ does not divide $3j-1$, these
equalities form a Hamiltonian cycle covering all the vertices of
$\Gamma_j$.  Then all the $n_k$ are equal, and $G=\Gamma_j(n/(3j-1))$.
\end{proof}

Figure~\ref{fig:cases-graphical} summarizes the regular graphs of interest
according to number of vertices and regular degree.  The horizontal line at
$2m/n^2=2/5$ represents the known result that minimum degree greater than
$2n/5$ in a triangle-free graph implies the graph is bipartite;
anything strictly above that line is bipartite.  On or below that line,
but above the horizontal line at $2m/n^2=10/29$, Lemma~\ref{lem:gamma-blowup}
implies only symmetric blowups of $\Gamma_i$ graphs (denoted by circles in
the figure) could be regular counterexamples to Conjecture~\ref{con:main}. 
And Corollary~\ref{cor:near-regular} implies that the curve labelled
``\eqref{eqn:ktwo-numerical} and \eqref{eqn:edges-stirling}''
eventually crosses (and then permanently remains above) the line at
$2m/n^2=2/5$, somewhere to the right of the region shown; it is approaching
an asymptote at $2m/n^2\approx 0.41 > 2/5$, and therefore the number of
$\Gamma_i(t)$ to consider is finite.

\begin{figure}
\centering\begin{tikzpicture}[>=latex']
  \draw[thick,bigah] (-0.5,0) -- (10.5,0);
  \draw[thick,bigah] (0,-0.5) -- (0,14.5);
  \node at (-1,11) {$\displaystyle \frac{2m}{n^2}$};
  \node at (5,-1) {$n$};
%
  \foreach \y in {14,9,6.727273,5.428571,4.588235,4,3.565217,%
     3.230769,2.965517,0.666667}
    {\draw (-0.15,\y) -- (0,\y);}
  \node[anchor=east] at (-0.2,14)
    {\small $2/5$};
  \node[anchor=east] at (-0.2,9)
    {\small $3/8$};
  \node[anchor=east] at (-0.2,6.727273)
    {\small $4/11$};
  \node[anchor=east] at (-0.2,5.428571)
    {\small $5/14$};
  \node[anchor=east] at (-0.2,4.588235)
    {\small $6/17$};
  \node[anchor=east] at (-0.2,4)
    {\small $7/20$};
  \node[anchor=east] at (-0.2,3.565217)
    {\small $8/23$};
  \node[anchor=east] at (-0.2,3.230769)
    {\small $9/26$};
  \node[anchor=east] at (-0.2,2.965517)
    {\small $10/29$};
  \node[anchor=east] at (-0.2,0.666667)
    {\small $1/3$};
%
  \foreach \x in {10,20,...,100}
    {\draw ($(\x/10,0)$) -- ($(\x/10,-0.15)$);
     \node[anchor=north] at ($(\x/10,-0.2)$) {\small \x};}
%
  \draw[densely dashed] (0,14)--(10.5,14);
  \draw[densely dashed] (0,2.965517)--(10.5,2.965517);
  \draw[densely dashed] (0,0.666667)--(10.5,0.666667);
  \node[anchor=south east] at (10,14) {$\uparrow\chi=2$};
  \node[anchor=south east] at (10,2.965517) {$\uparrow\chi\le3$};
  \node[anchor=south east] at (10,0.666667) {$\uparrow\chi\le4$};
%
  \draw[densely dashed] ($({44.8985/10},{(0.33003-0.33)*200})$)
  \foreach \x/\y in {
    46.3706/0.331732, 
    47.9176/0.333433, 
    49.6439/0.335235, 
    51.3645/0.336937, 
    53.2884/0.338739, 
    55.2102/0.34044, 
    57.1196/0.342042, 
    59.2636/0.343744, 
    61.5363/0.345445, 
    63.9494/0.347147, 
    66.3593/0.348749, 
    68.9164/0.35035, 
    71.8089/0.352052, 
    74.6465/0.353619, 
    77.607/0.355155, 
    80.8987/0.356757, 
    83.967/0.358158, 
    87.468/0.35966, 
    91.211/0.361161, 
    95.1515/0.362638, 
    99.2287/0.364064, 
    100.101/0.364357 
  } { -- ($({\x/10},{(\y-0.33)*200})$) };
  \node[anchor=south] at (9,7.5)
    {\eqref{eqn:ktwo-numerical} and \eqref{eqn:edges-stirling}};
  \draw[bigah] (9,7.5) -- ($({94.831/10},{(0.362525-0.33)*200})$);
%
  \draw[densely dotted] ($({14/10},{(0.336735-0.33)*200})$)
  \foreach \x/\y in {
    15/0.328889,
    16/0.343750,
    17/0.339100,
    18/0.345679,
    19/0.343490,
    20/0.350000,
    21/0.344671,
    22/0.351240,
    23/0.347826,
    24/0.354167,
    25/0.348800,
    26/0.357988,
    27/0.353909,
    28/0.359694,
    29/0.354340,
    30/0.360000,
    31/0.357960,
    32/0.363281,
    33/0.359963,
    34/0.363322,
    35/0.360816,
    36/0.365741,
    37/0.363769,
    38/0.367036,
    39/0.364234,
    40/0.368750,
    41/0.366449,
    42/0.369615,
    43/0.366685,
    44/0.370868,
    45/0.368395,
    46/0.372401,
    47/0.369398,
    48/0.373264,
    49/0.370679,
    50/0.374400,
    51/0.371396,
    52/0.375000,
    53/0.372375,
    54/0.375857,
    55/0.373554,
    56/0.376276,
    57/0.374269,
    58/0.377527,
    59/0.375180,
    60/0.377778,
    61/0.375705,
    62/0.378772,
    63/0.376921,
    64/0.379395,
    65/0.377278,
    66/0.379706,
    67/0.377812,
    68/0.380623,
    69/0.378492,
    70/0.380816,
    71/0.379290,
    72/0.381559,
    73/0.379809,
    74/0.382031,
    75/0.380444,
    76/0.382618,
    77/0.380840,
    78/0.382972,
    79/0.381349,
    80/0.383437,
    81/0.381954,
    82/0.383998,
    83/0.382349,
    84/0.384354,
    85/0.382837,
    86/0.384803,
    87/0.383142,
    88/0.385072,
    89/0.383790,
    90/0.385679
  } { -- ($({\x/10},{(\y-0.33)*200})$) };
  \node[anchor=south] at (6.3,12)
    {\eqref{eqn:ktwo-exact} and \eqref{eqn:edges-pifunc}};
  \draw[bigah] (6.3,12) -- ($({70/10},{(0.380816-0.33)*200})$);
%
  \draw[densely dashed] ($({100.707/10},{(0.350036-0.33)*200})$)
  \foreach \x/\y in {
    93.6566/0.351625,
    89.4778/0.352714,
    84.4704/0.354221,
    79.5556/0.355971,
    75.2655/0.357739,
    70.2393/0.360251,
    66.156/0.362764,
    62.7361/0.365276,
    59.9181/0.367789,
    57.5493/0.370302,
    55.5184/0.372814,
    53.83/0.375327,
    52.4245/0.377839,
    51.3535/0.380128,
    50.2256/0.382864,
    49.3919/0.385377,
    48.7269/0.387889,
    48.2133/0.390402,
    47.8379/0.392915,
    47.5908/0.395427,
    47.4647/0.39794,
    47.4548/0.400452,
    47.558/0.402965
  } { -- ($({\x/10},{(\y-0.33)*200})$) };
  \node[anchor=south] at (8,9)
    {\eqref{eqn:ktwo-numerical} and \eqref{eqn:hmorph-stirling}};
  \draw[bigah] (8,9) -- ($({66.156/10},{(0.362764-0.33)*200})$);
%
  \foreach \x in {5,10,...,45}
    {\draw[fill=white] (\x/10,14) circle[radius=0.1];}
  \foreach \x in {8,16,...,48}
    {\draw[fill=white] (\x/10,9) circle[radius=0.1];}
  \foreach \x in {11,22,...,55}
    {\draw[fill=white] (\x/10,6.727273) circle[radius=0.1];}
  \foreach \x in {14,28,...,70}
    {\draw[fill=white] (\x/10,5.428571) circle[radius=0.1];}
  \foreach \x in {17,34,...,68}
    {\draw[fill=white] (\x/10,4.588235) circle[radius=0.1];}
  \foreach \x in {20,40,60}
    {\draw[fill=white] (\x/10,4) circle[radius=0.1];}
  \foreach \x in {23,46}
    {\draw[fill=white] (\x/10,3.565217) circle[radius=0.1];}
  \foreach \x in {26,52}
    {\draw[fill=white] (\x/10,3.230769) circle[radius=0.1];}
%
%
\end{tikzpicture}
\caption{Cases and bounds.}
\label{fig:cases-graphical}
\end{figure}
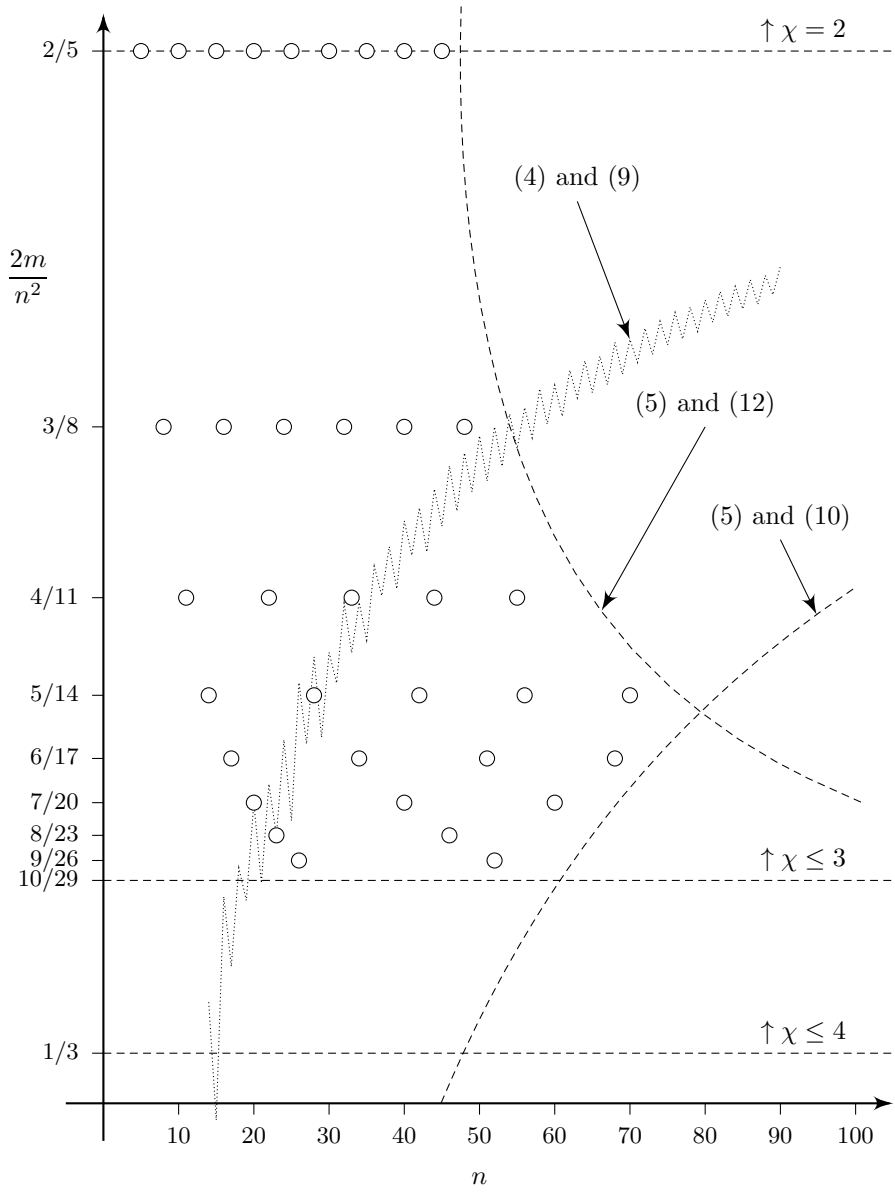

The bounds \eqref{eqn:edges-stirling} and \eqref{eqn:hmorph-stirling}
complement each other, as shown in Figure~\ref{fig:cases-graphical}; the
first works well for $\Gamma_i$ with relatively large $i$ and the second
works well with relatively small $i$.  Applying both, we can exclude all
blowups of $\Gamma_i$ for $2\le i \le 9$ except these:
$\Gamma_2(t)$ for $t\le 9$; $\Gamma_3(t)$ for $t\le 6$;
$\Gamma_4(t)$ for $t\le 5$; $\Gamma_5(t)$ for $t\le 5$;
$\Gamma_6(t)$ for $t\le 4$; $\Gamma_7(t)$ for $t\le 3$;
$\Gamma_8(t)$ for $t\le 2$; $\Gamma_9(t)$ for $t\le 2$.

By comparing \eqref{eqn:ktwo-exact} and \eqref{eqn:edges-pifunc}, which are
tighter but not closed-form versions of \eqref{eqn:edges-stirling} and
\eqref{eqn:hmorph-stirling}, we can exclude a few more.  This comparison is
shown by the zigzag dotted line in the figure; it assumes roughly the same
shape and is tending to the same asymptote as the curve for
\eqref{eqn:edges-stirling} and \eqref{eqn:hmorph-stirling}, because it comes
from the same calculation.  The zigzag pattern seems to result from parity
effects in \eqref{eqn:ktwo-exact}.  Although we conjecture that $T(n,2)$ is
cycle-maximal for both even and odd $n$, $T(n,2)$ is Hamiltonian only for
even $n$.  With odd $n$, there is always at least one vertex not included in
each cycle.  The fact that maximal-length cycles are a little shorter, and
therefore less numerous, when $n$ is odd makes $T(n,2)$ relatively poor in
cycles for odd $n$ overall, because almost all cycles are maximal-length or
very close.  The bound \eqref{eqn:edges-pifunc} has no special dependence on
parity, and so the gap between it and \eqref{eqn:ktwo-exact} tends to be
narrower for odd $n$, creating the zigzag pattern.  Using this bound allows
us to eliminate as possibilities $\Gamma_4(4)$, $\Gamma_4(5)$, all
$\Gamma_5(t)$ except $\Gamma_5$ itself, all $\Gamma_6(t)$ except $\Gamma_6$
itself, and all $\Gamma_7(t)$, $\Gamma_8(t)$, and $\Gamma_9(t)$. 
These computations, and the integer programming below, were performed
in the \eclipse\ constraint logic programming environment, which provides
easy access to backtracking search and large integer
arithmetic~\cite{Schimpf:Eclipse}.

Only 20 cases remain for maximal triangle-free graphs that are regular with
degree $>10n/29$.  All are eliminated by comparing
\eqref{eqn:ktwo-exact} with \eqref{eqn:perm-ibound} except
$\Gamma_2(1)=C_5$, which has one cycle and therefore is not cycle-maximal by
comparison with $T(5,2)=K_{2,3}$, which has three cycles.  The numerical
values for these cases are included in \ref{sec:numerics}.

At this point we have eliminated as possible counterexamples to
Conjecture~\ref{con:main} all regular graphs above the $2m/n^2=10/29$ line
in Figure~\ref{fig:cases-graphical}.  Then the comparison of
\eqref{eqn:ktwo-numerical} with \eqref{eqn:edges-stirling} eliminates all
regular graphs with $n>61$.  Any remaining regular counterexamples are
described by integers $n$ (number of vertices) and $\delta$ (regular degree)
satisfying these constraints:
\begin{equation}
  \begin{gathered}
    3 \le n \le 61 \, , \\
    2 \le \delta \le 10n/29 \, , \text{ and} \\
    m = n\delta/2 \text{ is an integer.}
  \end{gathered}
  \label{eqn:regular-intpro-csts}
\end{equation}

There are 428 pairs of $(n,\delta)$ satisfying
\eqref{eqn:regular-intpro-csts}.  All are excluded by comparing
\eqref{eqn:ktwo-exact} with \eqref{eqn:edges-pifunc}.  No more cases remain,
so the only regular graphs that can be cycle-maximal triangle-free are
of the form $T(n,2)$.  Finally, note that $T(n,2)$ is a regular
graph only when $n$ is even, so we have the following result.

\begin{theorem}
\label{thm:regular-triangle-free}
If $G$ is a regular cycle-maximal triangle-free graph with $n$ vertices,
then $n$ is even and $G$ is $K_{n/2,n/2}$.
\end{theorem}


\section{Cycles in near-regular triangle-free graphs}
\label{sec:near-regular}

If the minimum and maximum degrees in a graph differ by one, we will call
the graph \emph{near-regular}.  Note that this definition is strict:  regular
graphs are not near-regular.  When the minimum degree in a near-regular
graph is at most $2n/5$, then by counting $n-1$ vertices of degree
$(2n/5)+1$ and one vertex of degree $2n/5$, the maximum possible number of
edges is
\begin{equation*}
  \frac{n^2}{5} + \frac{n-1}{2} \, .
\end{equation*}

By substituting that into \eqref{eqn:edges-stirling} and comparing with
\eqref{eqn:ktwo-numerical}, any near-regular
cycle-maximal triangle-free graph that is not $T(n,2)$
can have at most 804 vertices.

To any near-regular graph $G$ we can assign the integer variables $n$ (number
of vertices); $m$ (number of edges); $\delta$ and $\Delta$ (the lower and
higher degrees respectively); and $n_\delta$ and $n_\Delta$ (number of low
and high-degree vertices respectively).  This collection of variables
is redundant, but
naming them all explicitly makes the constraints simpler.  With
the upper bound of $804$ vertices, and comparing \eqref{eqn:ktwo-exact} with
\eqref{eqn:edges-pifunc}, the following constraints apply to any
near-regular triangle-free graph that could be a counterexample to
Conjecture~\ref{con:main}.
\begin{equation}
  \begin{gathered}
    4\le n \le 804, \\
    n=n_\delta+n_\Delta, \, n_\delta>0, \, n_\Delta>0, \\
    2\le \delta\le \frac{2n}{5}, \, \Delta=\delta+1, \\
    m=\frac{1}{2}n_\delta\delta+\frac{1}{2}n_\Delta\Delta, \text{ and} \\
    \Pi(n,m) \frac{n^2}{8} \ge
      \sum_{k=2}^{\lfloor n/2 \rfloor} \frac{\lfloor n/2 \rfloor!\lceil n/2
        \rceil!}{2k(\lfloor n/2 \rfloor-k)!(\lceil n/2 \rceil-k)!} \, .
  \end{gathered}
  \label{eqn:nearreg-csts}
\end{equation}

By computer search with \eclipse~\cite{Schimpf:Eclipse}, $n\le 435$; and we
can obtain tighter bounds on $n$ for specific classes of graphs by further
constraining the minimum degree.
\begin{itemize}
  \item If $G$ is not homomorphic to $\Gamma_2$, $\delta\le 3n/8$
    and then $n\le 91$.
  \item If $G$ is not homomorphic to $\Gamma_3$, $\delta\le 4n/11$
    and then $n\le 61$
  \item If $G$ is not homomorphic to $\Gamma_4$, $\delta\le 5n/14$
    and then $n\le 51$.
  \item If $G$ is not homomorphic to $\Gamma_5$, $\delta\le 6n/17$
    and then $n\le 51$.
  \item If $G$ is not homomorphic to $\Gamma_6$, $\delta\le 7n/20$
    and then $n\le 43$.
  \item If $G$ is not homomorphic to $\Gamma_7$, $\delta\le 8n/23$
    and then $n\le 35$.
  \item If $G$ is not 3-colourable, $\delta\le 10n/29$ and then $n\le 35$.
  \item If $G$ is not 4-colourable, $\delta\le n/3$ and then $n\le 33$.
\end{itemize}

The same kind of argument used in Lemma~\ref{lem:gamma-blowup} can be used
to show that a not necessarily uniform blowup of a $\Gamma_i$ graph which is
near-regular obeys narrow bounds on its partition sizes.  The following
lemma gives the details for the case of $\Gamma_2=C_5$.

\begin{lemma}\label{lem:nearreg-gtwo-blowup}
If a near-regular graph $G$ is maximal triangle-free, homomorphic to
$\Gamma_2$, and not bipartite, then $G=\Gamma_2(n_1,n_2,n_3,n_4,n_5)$
with $n_2\le n_1+2$, $n_3\le n_1+1$, $n_4\le n_1+1$, and $n_5\le n_1+2$; and
therefore it is a subgraph of $\Gamma_2(\lfloor (n+6)/5 \rfloor)$.
\end{lemma}

\begin{proof}
If $G$ is maximal triangle-free and homomorphic to $\Gamma_2$,  then there
exist nonnegative integers $n_1,\ldots,n_5$, summing to $n$, so that
$G=\Gamma_2(n_1, \ldots, n_5)$.  If $G$ is not bipartite, then these are all
positive; and they cannot all be the same for the graph to be strictly
near-regular.  Therefore $n$ is at least $6$.

Let $v_1$, $v_2$, $v_3$, $v_4$, and $v_5$ be the vertices of $\Gamma_2$. 
Their neighbourhoods are respectively $\{v_3,v_4\}$, $\{v_4,v_5\}$,
$\{v_1,v_5\}$, $\{v_1,v_2\}$, and $\{v_2,v_3\}$.  The degree of vertices in
$G$ mapped by the homomorphism to any given vertex in $\Gamma_2$ is
equal to the sum of the sizes of sets of vertices in $G$ mapped to that
vertex's neighbours.  Therefore the following constraints hold:
\begin{gather*}
  |(n_3+n_4)-(n_4+n_5)| = |n_3-n_5| \le 1, \\
  |(n_4+n_5)-(n_1+n_5)| = |n_4-n_1| \le 1, \\
  |(n_1+n_5)-(n_1+n_2)| = |n_5-n_2| \le 1, \\
  |(n_1+n_2)-(n_2+n_3)| = |n_1-n_3| \le 1, \\
  |(n_2+n_3)-(n_3+n_4)| = |n_2-n_4| \le 1. \\
\end{gather*}
Let $n_1$ be the least of the $n_k$; then
$n_2\le n_1+2$, $n_3\le n_1+1$, $n_4\le n_1+1$, and $n_5\le n_1+2$.

Up to symmetry, there are nine cases for near-regular
$\Gamma_2(n_1,n_2,n_3,n_4,n_5)$ obeying the above constraints:
\begin{itemize}
  \item $G=\Gamma_2(n_1,n_1,n_1,n_1,n_1+1)$; then $n\equiv 1 \pmod{5}$
    and $G$ is a subgraph of $\Gamma_2(n+4)$.
  \item $G=\Gamma_2(n_1,n_1,n_1,n_1+1,n_1+1)$; then $n\equiv 2 \pmod{5}$
    and $G$ is a subgraph of $\Gamma_2(n+3)$.
  \item $G=\Gamma_2(n_1,n_1,n_1+1,n_1,n_1+1)$; then $n\equiv 2 \pmod{5}$
    and $G$ is a subgraph of $\Gamma_2(n+3)$.
  \item $G=\Gamma_2(n_1,n_1,n_1+1,n_1+1,n_1+1)$; then $n\equiv 3 \pmod{5}$
    and $G$ is a subgraph of $\Gamma_2(n+2)$.
  \item $G=\Gamma_2(n_1,n_1+1,n_1+1,n_1,n_1+1)$; then $n\equiv 3 \pmod{5}$
    and $G$ is a subgraph of $\Gamma_2(n+2)$.
  \item $G=\Gamma_2(n_1,n_1+1,n_1+1,n_1+1,n_1+1)$; then $n\equiv 4 \pmod{5}$
    and $G$ is a subgraph of $\Gamma_2(n+1)$.
  \item $G=\Gamma_2(n_1,n_1+1,n_1+1,n_1,n_1+2)$; then $n\equiv 4 \pmod{5}$
    and $G$ is a subgraph of $\Gamma_2(n+6)$.
  \item $G=\Gamma_2(n_1,n_1+1,n_1+1,n_1+1,n_1+2)$; then $n\equiv 0 \pmod{5}$
    and $G$ is a subgraph of $\Gamma_2(n+5)$.
  \item $G=\Gamma_2(n_1,n_1+2,n_1+1,n_1+1,n_1+2)$; then $n\equiv 1 \pmod{5}$
    and $G$ is a subgraph of $\Gamma_2(n+4)$.
\end{itemize}

In all these cases, $G$ is a subgraph of $\Gamma_2(\lfloor (n+6)/5
\rfloor)$.
\end{proof}

Lemma~\ref{lem:nearreg-gtwo-blowup} brings down the upper bound on $n$ a
little for the case of graphs homomorphic to $\Gamma_2$: because $G$
homomorphic to $\Gamma_2$ can have no more cycles than its supergraph
$\Gamma_2(\lfloor (n+6)/5\rfloor)$, we can compare \eqref{eqn:ktwo-exact}
for $n$ vertices with \eqref{eqn:hmorph-qn} for $5\lfloor (n+6)/5 \rfloor$
vertices, and find that for $G$ near-regular, cycle-maximal triangle-free,
and homomorphic to $\Gamma_2$ but not bipartite, $n\le 184$.

If we extend the constraint program \eqref{eqn:nearreg-csts} to include
separate variables for $n_1$, $n_2$, $n_3$, $n_4$, and $n_5$, with the
constraints on them given by Lemma~\ref{lem:nearreg-gtwo-blowup} and the new
bound $n\le 184$, we can generate an exhaustive list of the $\Gamma_2$
blowups that remain as possible counterexamples to
Conjecture~\ref{con:main}.  Comparing \eqref{eqn:ktwo-exact} with
\eqref{eqn:perm-ibound} for these cases eliminates all of them except the
three graphs shown in Figure~\ref{fig:gtwo-exceptions}: $\Gamma_2(1, 2, 1,
1, 2)$, $\Gamma_2(1, 2, 2, 1, 3)$, and $\Gamma_2(1, 3, 2, 2, 3)$.  Note that
the order of indices in $\Gamma_2$ and thus the order of indices in the
blowup notation is not consecutive around the five-cycle: $v_1$ in
$\Gamma_2$, under the definition, is adjacent to $v_3$ and $v_4$.  These
graphs are small enough that we can count the cycles exactly; none have as
many cycles as the bipartite Tur\'{a}n graph with the same number of
vertices.

\begin{figure}
  \centering
  \begin{tikzpicture}[scale=0.9]
    \begin{scope}[xshift=-5cm]
      \node[draw,circle] (a1) at (90:1.5) {~};
      \node[draw,circle] (b1) at (18:1.5) {~} edge (a1);
      \node[draw,circle] (c1) at (-54:1) {~} edge (b1);
      \node[draw,circle] (c2) at (-54:2) {~} edge (b1);
      \node[draw,circle] (d1) at (-126:1) {~} edge (c1) edge (c2);
      \node[draw,circle] (d2) at (-126:2) {~} edge (c1) edge (c2);
      \node[draw,circle] (e1) at (162:1.5) {~}
        edge (a1) edge (d1) edge (d2);
      \node at (0,-2.5) {$\Gamma_2(1, 2, 1, 1, 2)$};
    \end{scope}
    \begin{scope}
      \node[draw,circle] (a1) at (90:1.5) {~};
      \node[draw,circle] (b1) at (18:1.5) {~} edge (a1);
      \node[draw,circle] (c1) at (-54:1) {~} edge (b1);
      \node[draw,circle] (c2) at (-54:2) {~} edge (b1);
      \node[draw,circle] (d1) at (-126:1) {~} edge (c1) edge (c2);
      \node[draw,circle] (d2) at (-126:1.5) {~} edge (c1) edge (c2);
      \node[draw,circle] (d3) at (-126:2) {~} edge (c1) edge (c2);
      \node[draw,circle] (e1) at (162:1) {~}
        edge (a1) edge (d1) edge (d2) edge (d3);
      \node[draw,circle] (e2) at (162:2) {~}
        edge (a1) edge (d1) edge (d2) edge (d3);
      \node at (0,-2.5) {$\Gamma_2(1, 2, 2, 1, 3)$};
    \end{scope}
    \begin{scope}[xshift=5cm]
      \node[draw,circle] (a1) at (90:1.5) {~};
      \node[draw,circle] (b1) at (18:1) {~} edge (a1);
      \node[draw,circle] (b2) at (18:2) {~} edge (a1);
      \node[draw,circle] (c1) at (-54:1) {~} edge (b1) edge (b2);
      \node[draw,circle] (c2) at (-54:1.5) {~} edge (b1) edge (b2);
      \node[draw,circle] (c3) at (-54:2) {~} edge (b1) edge (b2);
      \node[draw,circle] (d1) at (-126:1) {~} edge (c1) edge (c2) edge (c3);
      \node[draw,circle] (d2) at (-126:1.5) {~} edge (c1) edge (c2) edge (c3);
      \node[draw,circle] (d3) at (-126:2) {~} edge (c1) edge (c2) edge (c3);
      \node[draw,circle] (e1) at (162:1) {~}
        edge (a1) edge (d1) edge (d2) edge (d3);
      \node[draw,circle] (e2) at (162:2) {~}
        edge (a1) edge (d1) edge (d2) edge (d3);
      \node at (0,-2.5) {$\Gamma_2(1, 3, 2, 2, 3)$};
    \end{scope}
  \end{tikzpicture}
  \caption{Near-regular graphs homomorphic to $\Gamma_2$ and not ruled out
    by comparing \eqref{eqn:ktwo-exact} with \eqref{eqn:perm-ibound}.}
  \label{fig:gtwo-exceptions}
\end{figure}
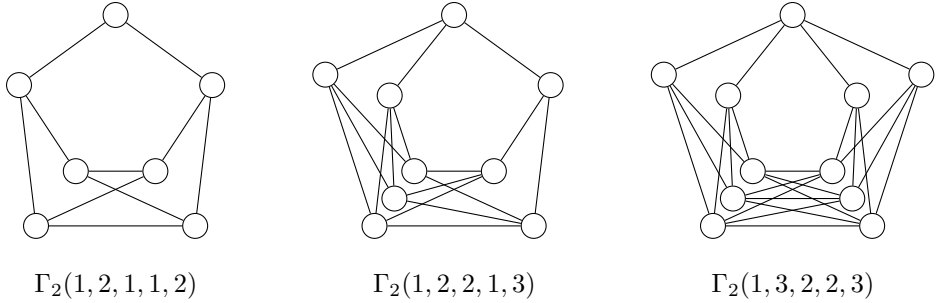

\begin{equation}
  \begin{aligned}
    c(\Gamma_2(1,1,2,1,2)) &= 15 \, , &
    c(\Gamma_2(1,2,2,1,3)) &= 216 \, , &
    c(\Gamma_2(1,3,2,2,3)) &= 3051 \, , \\
    c(T(7,2)) &= 42 \, , &
    c(T(9,2)) &= 660 \, , &
    c(T(11,2)) &= 15390 \, .
  \end{aligned}
  \label{eqn:gtwo-exceptions}
\end{equation}

These results suffice to establish the following theorem, which limits the
remaining possibilities for near-regular graphs that could be cycle-maximal
triangle-free.

\begin{theorem}\label{thm:near-reg-triangle-free}
If a graph $G$ with $n$ vertices and $m$ edges is cycle-maximal
triangle-free, its minimum and maximum degrees differ by exactly one, and
$G$ is not $T(n,2)$ with $n$ odd, then $n\le 91$, the minimum degree
in $G$ is at most $3n/8$, and $G$ is not homomorphic to $C_5$.
\end{theorem}

\begin{proof}
Suppose $G$ is a counterexample.  By comparing \eqref{eqn:edges-stirling}
with \eqref{eqn:ktwo-numerical}, $n\le 804$.  By solving the constraints
\eqref{eqn:nearreg-csts}, $n\le 435$.

For graphs homomorphic to $C_5$, by applying
Lemma~\ref{lem:nearreg-gtwo-blowup}, $n\le 184$.  Then by examining specific
graphs and comparing \eqref{eqn:ktwo-exact} with
\eqref{eqn:perm-ibound}, the three graphs shown in
Figure~\ref{fig:gtwo-exceptions} are the last remaining graphs homomorphic
to $C_5$, and \eqref{eqn:gtwo-exceptions} eliminates them.
For graphs not homomorphic to $C_5$: the minimum degree is at most
$3n/8$ because $G$ is maximal triangle-free.  Then by adding that constraint
to \eqref{eqn:nearreg-csts} and solving, $n\le 91$.
\end{proof}


\section{Algorithmic aspects of the upper bound calculation}
\label{sec:algorithm}

Lemma~\ref{lem:perm-bound} gives a bound~\eqref{eqn:perm-ibound} on number
of cycles in a graph in terms of the permanent of a matrix; that is the sum,
over all ways to choose one entry from each row and column, of the product
of the chosen entries.  Note that the matrix permanent is identical to the
matrix determinant except that in the determinant, each product is given a
sign depending on the parity of the permutation.  For the permanent, the
products are simply added.  Removing the signs has significant consequences
for the difficulty of computing the permanent: whereas computing the
determinant of an $n\times n$ matrix has the same asymptotic time complexity
as matrix multiplication (Cormen \emph{et al.}\ give this as an
exercise~\cite[Exercise 28.2--3]{CLRS}), permanent, like cycle counting, is
in general a $\#P$-complete problem, even when limited to 0-1
matrices~\cite{Valiant:Complexity}.

Solving one $\#P$-complete problem just to bound another is not
obviously useful.  However, the matrices for which we compute the
permanent to evaluate~\eqref{eqn:perm-ibound} are of a special form which
makes the computation much easier.
In this section we describe an algorithm
to compute such permanents with time complexity
having exponential dependence on $p$ (the number of vertices in $H$) but not
on $n$ (the number of vertices in $G$, and size of the matrix).

Ryser's formula~\cite{Ryser:Combinatorial} for the permanent of an $n\times
n$ matrix with entries $(a_{ij})$ is
\begin{equation}
    \perm (a_{ij}) =
      \hspace{-0.7em}\sum_{S\subseteq \{1,2,\ldots,n\}}\hspace{-0.7em}
      (-1)^{n-|S|} \prod_{i=1}^n \sum_{j \in S} a_{ij} \, .
    \label{eqn:ryser}
\end{equation}

Ryser's formula is a standard method for computing the permanent.  To
summarize it in words, the permanent is the sum over all subsets of the
columns of the matrix, of the product over all rows, of the sums of entries
in the chosen columns, with signs according to the parity of the size of the
subset.  The formula follows from applying the principle of inclusion and
exclusion to the permutation-based definition of permanent; and although
evaluating it has exponential time complexity because of the $2^n$ distinct
subsets of the columns, that is better than the factorial time complexity of
examining each permutation separately.

Suppose $A$ is an $n \times n$ binary matrix of the following form:
\begin{equation*}
  \begin{pmatrix}
    I_{n_1} & h_{12}J_{n_1n_2}          & \ldots & h_{1p}J_{n_1n_p} \\
    h_{21}J_{n_2n_1} & I_{n_2}          & \ldots & h_{2p}J_{n_2n_p} \\
    \vdots & \vdots &                     \ddots & \vdots \\
    h_{p1}J_{n_pn_1} & h_{p2}J_{n_pn_2} & \ldots & I_{n_p}
  \end{pmatrix} \, .
\end{equation*}
The rows are divided into $p$ blocks with sizes $n_1$, $n_2$, \ldots, $n_p$,
with $n=n_1+n_2+\cdots+n_p$.  The columns are divided into the same pattern
of blocks, giving the matrix an overall structure of $p$ blocks by $p$
blocks, with square blocks along the main diagonal but the other
blocks not necessarily square.  Furthermore, the blocks along the diagonal
of $A$ are identity matrices $I_{n_i}$ and the other blocks are of the form
$h_{ij}J_{n_in_j}$ with $h_{ij}\in\{0,1\}$; that is, blocks of all zeros or
all ones.  This is the form of the matrix for which we calculate the
permanent to evaluate~\eqref{eqn:perm-ibound}.

Observe that because of the block structure, many choices of the subset $S$
in~\eqref{eqn:ryser} will produce the same product of row sums.  The inside
of the first summation in~\eqref{eqn:ryser}, for matrices in the form we
consider, depends on how many columns are chosen from each block, but not
which ones.  If we let $k_i$ for $i \in \{1,2,\ldots,p\}$ be the number of
columns chosen in block $i$, then we can sum over the choices of all the
$k_i$ rather than the choices of $S$, using binomial coefficients to count
the number of choices of $S$ for each choice of all the $k_i$.  Furthermore,
the innermost sum need only contain $p$ terms for the block columns rather
than $n$ for the matrix columns, because we can collapse the sum within a
block of columns into $0$ for a block of all zeros; the number of columns
selected from the block for a block of all ones; or either $0$ or $1$ for an
identity-matrix block depending on whether we are in a row corresponding to
a selected column.  The product, similarly, only requires $2p$ factors,
raised to the appropriate powers, for the block rows and the choice of
``selected'' or ``not selected'' matrix rows; not $n$ possibilities for all
the matrix rows.  Algorithm~\ref{alg:permanent} gives pseudocode for the
calculation.

\begin{algorithm}
\caption{}\label{alg:permanent}
\begin{algorithmic}
  \STATE $result \gets 0$
  \FORALL{integer vectors $\langle k_1,k_2,\ldots,k_p \rangle$
    such that $0\le k_i \le n_i$}
    \STATE $cprod \gets 1$
    \FOR{$row=1$ \TO $p$}
      \STATE $rsum \gets 0$
      \FOR{$col=1$ \TO $p$}
        \IF{$row\ne col$ \AND $h[row,col]=1$}
          \STATE $rsum \gets rsum + k_{col}$
        \ENDIF
      \ENDFOR
      \STATE $cprod \gets cprod \cdot (rsum+1)^{k_{row}}
        \cdot rsum^{n_{row}-k_{row}}$
    \ENDFOR
    \STATE $result \gets result+cprod \cdot
      \prod_{i=1}^p (-1)^{n_i-k_i}\binom{n_i}{k_i}$
  \ENDFOR
  \RETURN{$result$}
\end{algorithmic}
\end{algorithm}

There are $\prod_{i=1}^p (n_i+1)$ choices for the vector $\langle
k_1,k_2,\ldots,k_p \rangle$; because equal division is the worst case,
that is $O(((n/p)+1)^p)$.
For each such vector, the inner loops do
$O(p^2)$ operations, giving the following result.

\begin{theorem}
There exists an algorithm to compute the permanent of a matrix $A$ in
$O(p^2((n/p)+1)^p)$ integer arithmetic operations if $A$ is an $n\times n$
matrix divided into $p$ blocks by $p$ blocks, not necessarily all of the
same size, in which the blocks along the main diagonal are identity matrices
and the other blocks each consist of all zeros or all ones.
\end{theorem}

When $p=n$, the case of general unblocked $n\times n$ matrices, this time
bound reduces to $O(n^22^n)$, which is the same as a straightforward
implementation of Ryser's formula.
Note that we describe the time complexity in terms of
``integer arithmetic operations.'' The value of the permanent can be on the
order of the factorial of the number of vertices $n$, in which case
representing it takes $O(n)$ words of $O(\log n)$ bits each.  We cannot do
arithmetic on such large numbers in constant time in the standard RAM model
of computation.  However, including an extended analysis here of the cost of
multiple-precision arithmetic would make the presentation more confusing
without providing any deeper insight into how the algorithm works.  Thus
we do the analysis in the unit cost model, with the caution that the
cost of arithmetic may be non-constant in practice and should be
considered when implementing the algorithm.
Even if our model does not include ``binomial coefficient'' as a
primitive constant-time operation, we can first build a
table of $k!$ for $k$ from 1 to $n$ with $O(n)$ multiplications, then
calculate $\binom{n}{k}$ as $n!/k!(n-k)!$ with three table lookups;
time and space to build the table are lower order than the overall cost
of Algorithm~\ref{alg:permanent}.

For the proofs in the previous sections, we implemented this algorithm in
the \eclipse\ language~\cite{Schimpf:Eclipse} with no particular effort to
optimize it, and found that the cost of calculating permanents to bound
cycle counts was comparable to the cost of the
integer programming to find the graphs in the first place, typically a few
CPU seconds per graph for small cases, up to a few hours for the largest
cases of interest.


\section{Conclusions and future work}
\label{sec:conclusion}

Conjecture~\ref{con:main} postulates that the bipartite Tur\'an graphs
achieve the maximum number of cycles among all triangle-free graphs. 
Depending on the parity of $n$, $T(n,2)$ is either regular or near-regular;
and we have ruled out all regular graphs and all but a finite number of
near-regular graphs as potential counterexamples to
Conjecture~\ref{con:main}.  It appears that our current techniques might be
extended to cover a few more of the near-regular cases by proving results
like Lemma~\ref{lem:nearreg-gtwo-blowup} for $\Gamma_3$, $\Gamma_4$, and so
on.  Each one reduces the maximum value of $\delta(G)/n$, and therefore the
maximum value of $n$, for which counterexamples could exist.

However, even if we could do this for all $\Gamma_i$, and extend the theory
to cover $4$-chromatic graphs too using the ``Vega graph'' classification
results of Brandt and Thomass\'e~\cite{Brandt:Dense}, potential
counterexamples with as many as 30 vertices would remain, and too many of
them to exhaustively enumerate as we did in the case of regular graphs. 
Similar issues apply even more strongly to graphs with $\Delta(G)-\delta(g)$
a constant $k>1$, even though by Corollary~\ref{cor:near-regular}, the
number of possible counterexamples is finite for any fixed $k$.  It seems
clear that to close these gaps will require a better theoretical
understanding of graphs with $\delta(G)$ less than but close to $n/3$, and
to finally prove Conjecture~\ref{con:main} we need better bounds for graphs
that are far from being regular.

When the girth increases beyond four the structure of cycle-maximal
graphs appears to change significantly.  In particular, they are not just
complete bipartite graphs with degree-two vertices inserted to increase the
lengths of the cycles.  For small values of $n$, our computer search showed
that most vertices in cycle-maximal graphs of fixed minimum girth $g \geq 5$
have degree three, with a few vertices of degree two and four present in
some cases.  Our preliminary examination of cycle-maximal graphs of girth
greater than four has yet to suggest any natural characterization of these
graphs, even when graphs are restricted to having regular degree.


\section*{Acknowledgements}
\label{sec:acknowledgements}

The authors thank the participants of the 2010 Workshop on Routing
in Merida with whom this problem was discussed for graphs of girth $g$
(for any general fixed value $g$); and an anonymous reviewer for a
correction to Lemma~\ref{lem:edge-bound}.

Work of the first and second authors was supported in part by the Natural
Sciences and Engineering Research Council of Canada (NSERC).


\appendix


\section{Example: the permanent bound for $C_5(2)$}
\label{sec:example}

This appendix demonstrates the permanent-based bound on number of cycles in
the graph $C_5(2)$, shown at left in Figure~\ref{fig:cfivetwo-example}. 
This graph comes up when trying to think of counterexamples to
Conjecture~\ref{con:main}: there is no instantly obvious reason for it to
have fewer cycles than $T(10,2)$, but in fact, it does have fewer cycles.

\begin{figure}
  \centering
  \begin{tikzpicture}
    \begin{scope}[xshift=-3cm]
      \node[draw,circle] (a1) at (18:1) {~};
      \node[draw,circle] (a2) at (18:2) {~};
      \node[draw,circle] (b1) at (90:1) {~} edge (a1) edge (a2);
      \node[draw,circle] (b2) at (90:2) {~} edge (a1) edge (a2);
      \node[draw,circle] (c1) at (162:1) {~} edge (b1) edge (b2);
      \node[draw,circle] (c2) at (162:2) {~} edge (b1) edge (b2);
      \node[draw,circle] (d1) at (234:1) {~} edge (c1) edge (c2);
      \node[draw,circle] (d2) at (234:2) {~} edge (c1) edge (c2);
      \node[draw,circle] (e1) at (306:1) {~}
        edge (a1) edge (a2) edge (d1) edge (d2);
      \node[draw,circle] (e2) at (306:2) {~}
        edge (a1) edge (a2) edge (d1) edge (d2);
      \node at (0,-2.25) {$G=C_5(2)$};
    \end{scope}
    \begin{scope}[xshift=3cm]
      \node[draw,circle] (a1) at (18:1.75) {$v_3$};
      \node[draw,circle] (b1) at (90:1.75) {$v_1$} edge (a1);
      \node[draw,circle] (c1) at (162:1.75) {$v_4$} edge (b1);
      \node[draw,circle] (d1) at (234:1.75) {$v_2$} edge (c1);
      \node[draw,circle] (e1) at (306:1.75) {$v_5$} edge (a1) edge (d1);
      \node at (0,-2.25) {$H=C_5$};
    \end{scope}
    \node at (0,0) {$\Rightarrow$};
  \end{tikzpicture}
  \caption{Graphs for the permanent bound example.}\label{fig:cfivetwo-example}
\end{figure}
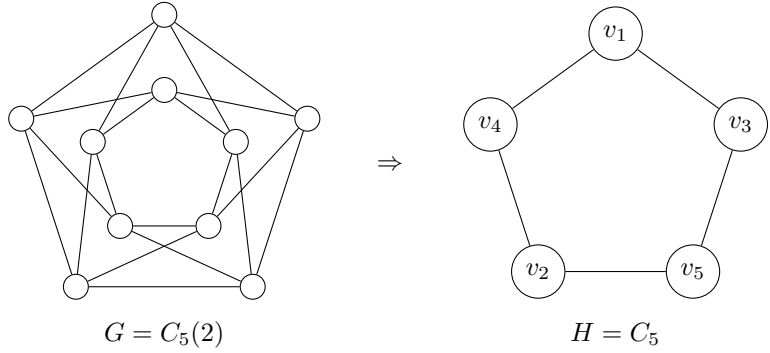

Let $G$ be the graph $C_5(2)$ and let $H$ be the graph $C_5$, which is the
same as $\Gamma_2$.  Figure~\ref{fig:cfivetwo-example} shows the
vertices of $H$ labelled as in the definition of $\Gamma_2$.  The adjacency
matrix of $H$ is
\begin{equation*}
  \begin{pmatrix}
    0 & 0 & 1 & 1 & 0 \\
    0 & 0 & 0 & 1 & 1 \\
    1 & 0 & 0 & 0 & 1 \\
    1 & 1 & 0 & 0 & 0 \\
    0 & 1 & 1 & 0 & 0
  \end{pmatrix} \, .
\end{equation*}

The graph $G$ is obtained by blowing up each vertex of $H$ into a two-vertex
independent set, and in the adjacency matrix that is equivalent to replacing
each element with a $2\times 2$ submatrix.  To apply the bound of
Lemma~\ref{lem:perm-bound}, we also add ones along the diagonal, giving this
modified version of the adjacency matrix of $G$:
\begin{equation*}
  \left(\begin{array}{cc|cc|cc|cc|cc}
    1 & 0 & 0 & 0 & 1 & 1 & 1 & 1 & 0 & 0 \\
    0 & 1 & 0 & 0 & 1 & 1 & 1 & 1 & 0 & 0 \\ \hline
    0 & 0 & 1 & 0 & 0 & 0 & 1 & 1 & 1 & 1 \\
    0 & 0 & 0 & 1 & 0 & 0 & 1 & 1 & 1 & 1 \\ \hline
    1 & 1 & 0 & 0 & 1 & 0 & 0 & 0 & 1 & 1 \\
    1 & 1 & 0 & 0 & 0 & 1 & 0 & 0 & 1 & 1 \\ \hline
    1 & 1 & 1 & 1 & 0 & 0 & 1 & 0 & 0 & 0 \\
    1 & 1 & 1 & 1 & 0 & 0 & 0 & 1 & 0 & 0 \\ \hline
    0 & 0 & 1 & 1 & 1 & 1 & 0 & 0 & 1 & 0 \\
    0 & 0 & 1 & 1 & 1 & 1 & 0 & 0 & 0 & 1
  \end{array}\right) \, .
\end{equation*}

The permanent of that $10\times 10$ matrix is 5753, so by
Lemma~\ref{lem:perm-bound}, taking the floor because the cycle count is an
integer, $C_5(2)$ contains at most 2876 cycles.  In fact, by exact count
$C_5(2)$ contains 593 cycles.  Both numbers are less than the 3940 cycles in
$T(10,2)$.


\section{Numerical results}
\label{sec:numerics}

Tables~\ref{tab:counts-lo} and~\ref{tab:counts-hi} list exact counts and
bounds on the number of cycles in various graphs, sorted by number of
vertices for easier comparisons.

\begin{table}
\caption{Cycle counts and bounds for various graphs, $n\le 30$.}
\label{tab:counts-lo}
\centering
\begin{tabular}{rrrc}
$G$ & $n$ & $c(G)$ & from \\ \hline
$\Gamma_2=C_5$ & 5 & 1 & obvious \\
$K_{2,3}$ & 5 & 3 & \eqref{eqn:ktwo-exact} \\
$\Gamma_3$ & 8 & $\le$130 & \eqref{eqn:perm-ibound}\\
$K_{4,4}$ & 8 & 204 & \eqref{eqn:ktwo-exact} \\
$\Gamma_2(2)$ & 10 & $\le$2~876 & \eqref{eqn:perm-ibound}\\
$K_{5,5}$ & 10 & 3~940 & \eqref{eqn:ktwo-exact} \\
$\Gamma_4$ & 11 & $\le$6~151 & \eqref{eqn:perm-ibound}\\
$K_{5,6}$ & 11 & 15~390 & \eqref{eqn:ktwo-exact} \\
$\Gamma_5$ & 14 & $\le$602~261 & \eqref{eqn:perm-ibound}\\
$K_{7,7}$ & 14 & 4~662~231 & \eqref{eqn:ktwo-exact} \\
$\Gamma_2(3)$ & 15 & $\le$12~782~394 & \eqref{eqn:perm-ibound}\\
$K_{7,8}$ & 15 & 24~864~588 & \eqref{eqn:ktwo-exact} \\
$\Gamma_3(2)$ & 16 & $\le$36~552~880 & \eqref{eqn:perm-ibound}\\
$K_{8,8}$ & 16 & 256~485~040 & \eqref{eqn:ktwo-exact} \\
$\Gamma_6$ & 17 & $\le$104~770~595 & \eqref{eqn:perm-ibound}\\
$K_{8,9}$ & 17 & 1~549~436~112 & \eqref{eqn:ktwo-exact} \\
$\Gamma_7$ & 20 & $\le$29~685~072~610 & \eqref{eqn:perm-ibound}\\
$\Gamma_2(4)$ & 20 & $\le$275~455~237~776 & \eqref{eqn:perm-ibound}\\
$K_{10,10}$ & 20 & 1~623~855~701~385 & \eqref{eqn:ktwo-exact} \\
$\Gamma_4(2)$ & 22 & $\le$3~544~330~396~616 & \eqref{eqn:perm-ibound}\\
$K_{11,11}$ & 22 & 177~195~820~499~335 & \eqref{eqn:ktwo-exact} \\
$\Gamma_3(3)$ & 24 & $\le$504~887~523~966~914 & \eqref{eqn:perm-ibound}\\
$K_{12,12}$ & 24 & 23~237~493~232~953~516 & \eqref{eqn:ktwo-exact} \\
$\Gamma_2(5)$ & 25 & $\le$19~610~234~100~506~750 & \eqref{eqn:perm-ibound}\\
$K_{12,13}$ & 25 & 205~717~367~581~496~628 & \eqref{eqn:ktwo-exact} \\
$\Gamma_5(2)$ & 28 & $\le$1~583~204~062~862~484~492 & \eqref{eqn:perm-ibound}\\
$K_{14,14}$ & 28 & 653~193~551~573~628~900~289 & \eqref{eqn:ktwo-exact} \\
$\Gamma_2(6)$ & 30 & $\le$3~664~979~770~718~930~748~156 & \eqref{eqn:perm-ibound}\\
$K_{15,15}$ & 30 & 136~634~950~180~317~224~866~335 & \eqref{eqn:ktwo-exact}
\end{tabular}
\end{table}

\begin{sidewaystable}
\caption{Cycle counts and bounds for various graphs, $n>30$.}
\label{tab:counts-hi}
\centering
\begin{tabular}{rrrc}
$G$ & $n$ & $c(G)$ & from \\ \hline
$\Gamma_3(4)$ & 32 & $\le$93~314~267~145~221~727~988~928 & \eqref{eqn:perm-ibound}\\
$K_{16,16}$ & 32 & 32~681~589~590~709~963~123~092~160 & \eqref{eqn:ktwo-exact} \\
$\Gamma_4(3)$ & 33 & $\le$472~536~908~624~040~051~159~801 & \eqref{eqn:perm-ibound}\\
$K_{16,17}$ & 33 & 380~842~679~006~967~756~257~282~880 & \eqref{eqn:ktwo-exact} \\
$\Gamma_2(7)$ & 35 & $\le$1~538~132~015~230~964~742~594~686~226 & \eqref{eqn:perm-ibound}\\
$K_{17,18}$ & 35 & 109~481~704~025~024~759~751~150~754~248 & \eqref{eqn:ktwo-exact} \\
$\Gamma_3(5)$ & 40 & $\le$121~876~741~093~584~265~201~282~594~275~138 & \eqref{eqn:perm-ibound}\\
$\Gamma_2(8)$ & 40 & $\le$1~295~546~973~219~341~717~643~333~826~977~344 & \eqref{eqn:perm-ibound}\\
$K_{20,20}$ & 40 & 350~014~073~794~168~154~275~473~348~323~458~540 & \eqref{eqn:ktwo-exact} \\
$\Gamma_2(9)$ & 45 & $\le$2~011~552~320~593~475~430~049~513~125~845~530~235~126 & \eqref{eqn:perm-ibound}\\
$K_{22,23}$ & 45 & 1~072~464~279~544~434~376~131~539~091~650~605~148~971~323 & \eqref{eqn:ktwo-exact} \\
$\Gamma_3(6)$ & 48 & $\le$765~658~164~243~897~411~689~143~843~074~192~950~614~512 & \eqref{eqn:perm-ibound}\\
$K_{24,24}$ & 48 & 18~847~819~366~080~117~996~802~964~862~587~612~140~097~642~544 & \eqref{eqn:ktwo-exact} \\
$\Gamma_2(10)$ & 50 & $\le$5~387~065~180~713~482~750~668~088~096~305~965~320~151~649~500 & \eqref{eqn:perm-ibound}\\
$K_{25,25}$ & 50 & 11~294~267~336~237~005~395~453~340~472~970~226~376~143~920~186~000 & \eqref{eqn:ktwo-exact} \\
$\Gamma_3(7)$ & 56 & $\le$17~877~864~251~518~595~245~276~779~749~582~885~338~633~210~045~796~098 & \eqref{eqn:perm-ibound}\\
$K_{28,28}$ & 56 & 3~883~426~377~993~747~808~177~077~817~275~217~253~080~577~404~858~001~996~940 & \eqref{eqn:ktwo-exact}
\end{tabular}
\end{sidewaystable}


\clearpage

\bibliographystyle{model1b-num-names}
\bibliography{cyclemax}

\end{document}